\documentclass[12pt]{amsart}
\usepackage{amsthm,amsmath,a4,amssymb,verbatim,url,enumerate}

\newtheorem{thm}{Theorem}[section]
\newtheorem{cor}[thm]{Corollary}
\newtheorem{lem}[thm]{Lemma}
\newtheorem{prop}[thm]{Proposition}
\theoremstyle{definition}
\newtheorem{defn}[thm]{Definition}
\theoremstyle{remark}
\newtheorem{que}[thm]{Question}
\newtheorem{rem}[thm]{Remark}
\newtheorem{ex}[thm]{Example}

\newcommand{\Aut}{\textnormal{Aut}}

\newcommand{\Ker}{\textnormal{Ker}}

\newcommand{\GL}{\textnormal{GL}}

\newcommand{\g}{\mathfrak{g}}

\newcommand{\mk}{\mathfrak}
\newcommand{\mka}{\mathfrak{a}}

\newcommand{\mkl}{\mathfrak{l}}
\newcommand{\mkh}{\mathfrak{h}}

\newcommand{\sys}{\textnormal{sys}}
\newcommand{\covol}{\textnormal{covol}}

\newcommand{\pack}{\textnormal{pack}}
\newcommand{\ger}{\textnormal{ger}}
\newcommand{\cni}{\textnormal{cni}}
\newcommand{\Car}{\textnormal{Car}}

\newcommand{\ot}{\otimes}

\newcommand{\sfX}{\mathsf{X}}
\begin{document}

\address{CNRS -- D\'epartement de Math\'ematiques, Universit\'e Paris-Sud, 91405 Orsay, France}

\email{yves.cornulier@math.u-psud.fr}

\title[Gradings, systolic growth and cohopfian groups]{Gradings on Lie algebras, systolic growth, and cohopfian properties of nilpotent groups}
\author{Yves Cornulier}%
\date{August 4, 2016}
\thanks{Supported by ANR 12-BS01-0003-01 {\em G\'eom\'etrie des sous-groupes}}


\begin{abstract}
We study the existence of various types of gradings on Lie algebras, such as Carnot gradings or gradings in positive integers, and prove that the existence of such gradings is invariant under extensions of scalars.

As an application, we prove that if $\Gamma$ is a finitely generated nilpotent group, its systolic growth is asymptotically equivalent to its word growth if and only if the Malcev completion of $\Gamma$ is Carnot.

We also characterize when $\Gamma$ is non-cohopfian, in terms of the existence of a non-trivial grading in non-negative integers, and deduce that this property only depends on its real (or even complex) Malcev completion. 
\end{abstract}
\maketitle





\section{Introduction}

The purpose of this paper is twofold: to discuss the existence of certain kinds of gradings on Lie algebras (and some more general algebras), and to apply the results to the study of several aspects of finitely generated nilpotent groups.

\subsection{Cohopfian properties and systolic growth}

This paper will study some properties in the case of finitely generated nilpotent groups. In this subsection, we introduce these properties in general.

\subsubsection{Systolic growth}\label{iisg}

Let $\Gamma$ be a finitely generated group, and endow it with the word metric with respect to some finite generating subset $S$. If $\Lambda\subset\Gamma$, define its systole $\sys_S(\Lambda)$ to be $\inf\{|g|_S:g\in\Lambda\smallsetminus\{1\}\}$ (which is $+\infty$ in case $\Lambda=\{1\}$). Define, following \cite{Gro} its {\em systolic growth} as the function $\sigma_{\Gamma,S}$ mapping $n$ to the smallest index of a subgroup of systole $\ge n$ (hence $+\infty$ if there is no such subgroup). Note that $\Gamma$ is residually finite if and only if $\sigma_{\Gamma,S}(n)<\infty$ for all $n$, and a standard argument shows that the asymptotic behavior (in the usual sense of growth of groups, see \S\ref{gelani}) of $\sigma_{\Gamma,S}$ does not depend on the choice of $S$; hence we call it the systolic growth of $\Gamma$. It is obviously asymptotically bounded below by the growth (precisely, $\sigma_{\Gamma,S}(2n+1)\ge b_{\Gamma,S}(n)$, where $b_{\Gamma,S}(n)$ is the cardinal of the $n$-ball). It is easy to see that $\Gamma$ and its finite index subgroups have asymptotically equivalent systolic growth. 

It is natural to wonder when the growth and systolic growth are equivalent. 
Gromov \cite[p.334]{Gro} provides a simple argument, based on congruence subgroups, showing that finitely generated subgroups of $\GL_d(\mathbf{Q})$
have at most exponential systolic growth (although he states a less general fact). Bou-Rabee and the author \cite{BC} actually prove that all finitely generated linear subgroups (i.e., with a faithful finite-dimensional linear representation over some field) have at most exponential systolic growth, and hence exactly exponential systolic growth when the growth is exponential. For finitely generated linear groups, this thus reduces the question of equivalence of growth and systolic growth to virtually nilpotent groups.

\begin{rem}\label{rgirth}
A notion closely related to systolic growth was introduced by Bou-Rabee and McReynolds \cite{BM} (apparently independently of \cite{Gro}), defining the {\em residual girth} of a group in the same way as the systolic growth above, but restricting to {\em normal} finite index subgroups. If we denote by $\sigma_{\Gamma,S}^\lhd(n)$ the resulting function, we obviously have
\[\sigma_{\Gamma,S}\le\sigma_{\Gamma,S}^\lhd\le\sigma_{\Gamma,S}!.\] The examples in \cite{BS} show that $\sigma_{\Gamma,S}$, for finitely generated residually finite groups, can be arbitrary large.

On the other hand, there is an exponential upper bound for the residual girth of finitely generated linear groups \cite{BC}.

Besides, we can define one more notion: if $\Lambda\subset\Gamma$, define its {\em normal systole} $\sys_S^\lhd(\Lambda)$ as the infimum of $|g|_S$, when $g$ ranges over {\em $\Gamma$-conjugates} of elements in $\Lambda\smallsetminus\{1\}$. Note that it has a geometric interpretation: let $\mathcal{G}(\Gamma,S)$ be the Cayley graph of $\Gamma$ with respect to $S$. While $\sys_S(\Lambda)$ is the length of the smallest non-trivial based combinatorial loop in the quotient $\Lambda\backslash\mathcal{G}(\Gamma,S)$ (where non-trivial means it does not lift to a loop in $\mathcal{G}(\Gamma,S)$), the normal systole $\sys_S^\lhd(\Lambda)$ is the length of the smallest non-trivial combinatorial loop (not necessarily based); of course when $\Lambda$ is normal, its normal systole equals its systole. We can then define the {\em uniform systolic growth} of $\Gamma$ as the function $\sigma^u_{\Gamma,S}$ mapping $n$ to the smallest index of a subgroup of normal systole $\ge n$. Thus clearly we have 
\[\sigma_{\Gamma,S}\le\sigma^u_{\Gamma,S}\le\sigma_{\Gamma,S}^\lhd.\]

I do not know examples for which the uniform systolic growth is not equivalent to the systolic growth; on the other hand simple examples show that it can fail to be equivalent to the residual girth, see Remark \ref{rgirth2}.
\end{rem}

\subsubsection{Cohopfian properties}

Recall that a group is {\em non-cohopfian} if it admits a non-surjective injective endomorphism, and {\em cohopfian} otherwise.
Let us say that a group $\Gamma$ is {\em dis-cohopfian} if it admits an injective endomorphism $\phi$ such that $\bigcap_{n\ge 0}\phi^n(\Gamma)=\{1\}$; such $\phi$ is called a dis-cohopf endomorphism. It appears, for a nontrivial group, as a strong negation of being cohopfian.

Let us consider an intermediate notion: we say that a group $\Gamma$ is {\em weakly dis-cohopfian} if it admits a sequence of subgroups $(\Gamma_n)$, all isomorphic to $\Gamma$, with $\Gamma_{n+1}$ contained in $\Gamma_n$ for all $n$, and $\bigcap_n\Gamma_n=\{1\}$. This is implied by dis-cohopfian, and implies, for a nontrivial group, non-cohopfian.

\begin{ex}
The group $\mathbf{Z}$ is dis-cohopfian. Slightly less trivially, the infinite dihedral group is dis-cohopfian. Actually, every nontrivial free product $A\ast B$ is dis-cohopfian (that non-trivial free products are non-cohopfian is well-known; whether they are always dis-cohopfian was asked to me by K.\ Bou-Rabee): indeed, fix nontrivial elements $a_0\in A$, $b_0\in B$ and consider the endomorphism $\phi$ defined by $a\mapsto a_0b_0ab_0^{-1}a_0^{-1}$ for $a\in A$ and $b\mapsto b_0a_0ba_0^{-1}b_0^{-1}$. Let $S$ be the generating set $A\cup B$ for $A\ast B$, and $|\cdot|$ the corresponding word length. Then this endomorphism formally maps reduced words to reduced words, and hence $|\phi(x)|=5|x|$ for all $x$; thus $|\phi^n(x)|=5^n|x|$ for all $x$ and it follows that $\bigcap_{n\ge 0}\mathrm{Im}(\phi^n)=\{1\}$.

The group $\mathbf{Z}\times\mathbf{Z}/2\mathbf{Z}$ is non-cohopfian but not weakly dis-cohopfian. Examples of groups that are weakly dis-cohopfian but not dis-cohopfian will be provided in Example \ref{wdchd}. These are the first such examples among finitely generated nilpotent groups. However, it seems from the discussion in \cite{NP} that they were aware of other examples among polycyclic groups of exponential growth.

Using the Frobenius endomorphism, it is also possible to find examples of dis-cohopfian groups with exponential growth, such as suitable finite index subgroups of $\textnormal{SL}_d(\mathbf{F}_p[t])$, or the lamplighter group $(\mathbf{Z}/p\mathbf{Z})\wr\mathbf{Z}$. 
\end{ex}

\begin{rem}
There are natural analogues of these cohopfian-like properties, where injective endomorphisms are required to have an image of finite index. For finitely generated nilpotent groups (and more generally for virtually polycyclic groups, or even finitely generated solvable groups of finite Pr\"ufer rank), injective endomorphisms automatically have an image of finite index, and thus the notions coincide. It is unknown (see \cite{NP}) if there exists a finitely generated group that is not virtually nilpotent, but admits an injective endomorphism $\phi$ with image of finite index, such that $\bigcap_{n\ge 0}\mathrm{Im}(\phi^n)=\{1\}$.
\end{rem}

\subsection{Gradings on Lie algebras}
\subsubsection{Main definitions and results}
This subsection is independent of the previous one. It introduces some notions of Lie algebras and some results about these notions, which will be used in the study of nilpotent groups in the next subsection, but purport to be of independent interest, notably for readers interested in the classification of finite-dimensional nilpotent Lie algebras.

We denote by $R$ a ground commutative ring (always assumed associative with unit); in most of the discussion, $R$ will be a field. We abbreviate ``field of characteristic zero" into ``$\mathbf{Q}$-field".

Given an abelian group $A$, recall that an $A$-grading of a Lie $R$-algebra $\g$ is a direct sum decomposition $\g=\bigoplus_{\alpha\in A}\g_\alpha$ where each $\g_\alpha$ is an $R$-submodule, and satisfying $[\g_\alpha,\g_\beta]\subset\g_{\alpha+\beta}$ for all $\alpha,\beta\in A$. 

Gradings are a convenient way to encode various actions: for instance, a grading in $\mathbf{Z}$ gives rise to an action of the multiplicative group $\GL_1(R)$ on $\g$, where $r\in\GL_1(R)$ acts by multiplication by $r^n$ on $\g_n$. Conversely, under suitable assumptions (algebraic action, $R$ a field, $\g$ of finite dimension), every $\GL_1$-action yields a $\mathbf{Z}$-grading of $\g$. Further conditions on the grading, as below, correspond to further conditions on the action.

Every Lie algebra has a trivial grading, namely with $\g=\g_0$. 
Here are three conditions on a $\mathbf{Z}$-grading of a Lie algebra $\g$:
\begin{itemize}
\item the grading is {\em non-negative}, that is, in the natural numbers $\mathbf{N}=\{0,1,2,\dots\}$ (i.e., $\g_{i}=0$ for all $i<0$). If $\g$ admits a non-negative non-trivial grading, we call $\g$ {\em semi-contractable};
\item the grading is {\em positive}, that is, in the positive natural numbers $\mathbf{N}^+$. If $\g$ admits a positive grading, we call $\g$ {\em contractable};
\item the Lie algebra $\g$ is generated by $\g_1$; we then say that the grading is {\em Carnot}. If a Lie algebra $\g$ admits a Carnot grading, we call $\g$ Carnot; if $\g$ is endowed with a Carnot grading, it is called a Carnot-graded Lie algebra (see \S\ref{nlaaaa} for counterexamples among nilpotent Lie algebras). 
\end{itemize}

Note that each of these conditions on the grading is implied by the next one; thus $\g$ Carnot implies $\g$ contractable, which implies (for $\g\neq 0$) $\g$ semi-contractable. Note that when $\g$ is a finitely generated $R$-module and is contractable, then it is nilpotent and actually our emphasis will be on finite-dimensional nilpotent Lie algebras over fields, especially of characteristic zero.

Being Carnot can be redefined in the following way: if $\g$ is an arbitrary Lie algebra over the commutative ring $R$, let $(\g^{(i)})_{i\ge 1}$ be its lower central series ($\g^{(1)}=\g$ and $\g^{(i+1)}=[\g,\g^{(i)}]$). 
The {\em associated Carnot-graded Lie algebra} $\mathrm{Car}(\g)$ is defined as $\bigoplus_{i\ge 1}\g^{(i)}/\g^{(i+1)}$ with the naturally induced bracket and grading; the Lie algebra $\g$ is {\em Carnot} if it is isomorphic (as a Lie $R$-algebra) to its associated Carnot-graded Lie algebra. A special feature of Carnot gradings is that they are all conjugate under $\Aut(\g)$; in particular two Carnot-graded Lie algebras are isomorphic as graded Lie algebras if and only if they are isomorphic as Lie algebras. Nevertheless, this is only uniqueness up to conjugacy, and Carnot and Carnot-graded should be distinguished; for instance, for a non-abelian Carnot-graded Lie algebra, the automorphism group as a Lie algebra is larger than the graded automorphism group (see Corollary \ref{autautgr} for a precise comparison).

Carnot Lie algebras are ubiquitous in the study of Lie algebras and the associated Lie and discrete groups. 
 For instance Pansu \cite{Pan} proved that any two quasi-isometric simply connected real Lie groups have isomorphic associated Carnot-graded real Lie algebras. The classification of various classes of nilpotent finite-dimensional Lie algebras also starts with the Carnot case: for instance Vergne \cite{Ver} classified the Carnot-graded $d$-dimensional Lie algebras of nilpotency length exactly $d-1$ over a field of characteristic $\neq 2$: for $d\ge 2$ there are 1 or 2 such Lie algebras, and 2 precisely when $d\ge 6$ is even; the corresponding classification in the non-Carnot case is out of reach in general, and only known in small dimension.

Turning back to gradings, there are further natural conditions not related to positivity, such as 
\begin{itemize}
\item The $\mathbf{Z}$-grading is {\em invertible} if $\g_0=\{0\}$. 
\end{itemize}

When the ground $\mathbf{Q}$-field is algebraically closed, these various conditions can be easily and conveniently characterized in terms of the maximal grading, see \S\ref{cagr}.

If $\g$ is a Lie algebra and $R'$ is a commutative $R$-algebra, then every grading on $\g$ induces a grading on the Lie $R'$-algebra $R'\ot_R\g$, which inherits any of the above conditions. It is natural to wonder whether conversely, under reasonable hypotheses, the existence of a grading on $R'\ot_R\g$ with additional properties implies the existence of a similar grading on $\g$.

Let us begin with a few simple counterexamples. The 3-dimensional Lie algebra $\mk{sl}_2(\mathbf{C})$ admits a grading in $\{-1,0,1\}$ with each component of dimension 1. On the other hand, the real Lie algebra $\mk{so}_3(\mathbf{R})$ admits no non-trivial $\mathbf{Z}$-grading (because all its self-derivations are inner and have eigenvalues in $i\mathbf{R}$), and $\mathbf{C}\ot_\mathbf{R}\mk{so}_3(\mathbf{R})$ and $\mk{sl}_2(\mathbf{C})$ are isomorphic as complex Lie algebras. More subtle counterexamples will be provided in the sequel, but let us begin with positive results. 

\begin{thm}[see Th.\ \ref{main2} and Th.\ \ref{stabex}]\label{beingca}
Being Carnot, contractable, semi-contractable are invariant under taking extensions of $\mathbf{Q}$-fields. That is, if $R=K$ is a $\mathbf{Q}$-field, $L$ is an extension of $K$ and $\g$ is a finite-dimensional Lie $K$-algebra, then $\g$ is Carnot (resp.\ contractable, resp.\ semi-contractable) if and only if $L\ot_K\g$ satisfies the same property as a Lie algebra over $L$.
\end{thm}

There is an analogy between Theorem \ref{beingca} and Sullivan's result \cite[Theorem 12.1]{Sul} that the notion of formality for a nilpotent minimal differential algebra is independent of the ground field of characteristic zero; however we are not aware of a link between these two facts. For the Carnot property and the extension $\mathbf{Q}\subset\mathbf{R}$, the question whether Carnot goes from $\mathbf{R}$ down to $\mathbf{Q}$ was raised in 1975 by Johnson \cite{Joh}; a positive solution was written by Dekimpe and Lee but their argument is mistaken (see Remark \ref{remdl}). 

Let us introduce a few ideals canonically associated to a finite-dimensional Lie algebra $\g$ over a $\mathbf{Q}$-field. 

\begin{defn}The CNI-radical, or (relatively) characteristically nilpotent radical $\cni(\g)$ of $\g$ is the intersection of all kernels of all semisimple self-derivations of $\g$. We say that $\g$ is {\em essentially flexible} if $\cni(\g)=\{0\}$, and {\em flexible} if it admits an invertible self-derivation.\end{defn}

This seems to be new notions. We say that an ideal is {\em characteristic} if it is invariant under all automorphisms (beware that there exist alternative definitions of characteristic ideals, for instance in the case of Lie algebra). The CNI-radical is a nilpotent characteristic ideal. Classically, $\g$ is called characteristically nilpotent if $\cni(\g)=\g$, or equivalently if every self-derivation of $\g$ is nilpotent; the smallest nonzero examples are 7-dimensional, see the survey by Ancochea and Campoamor \cite{AC}. By elementary arguments, the CNI-radical is well-behaved with respect to extensions, in the sense that $\cni(L\ot_K\g)=L\ot_K\cni(\g)$, and being essentially flexible or flexible is invariant by taking field extensions. Every characteristically nilpotent characteristic ideal is contained in $\cni(\g)$; however in a nilpotent Lie algebra, $\cni(\g)$ is not always characteristically nilpotent, see the example in \S\ref{cnincn}.

The CNI-radical is also the intersection of fixed point sets in $\g$ of all connected sub-tori in $\Aut(\g)$.

When the ground $\mathbf{Q}$-field is algebraically closed, being flexible is equivalent to the existence of an invertible $\mathbf{Z}$-grading. However, the latter property is not invariant under taking field extensions, see \S\ref{exdere}, which relies on a construction of Der\'e. Another property of Lie algebras over $\mathbf{Q}$ that does not behave well with respect to extensions is to be Anosov; see \S\ref{anoni}.

\begin{defn}The {\em uncontractable radical} $\cni^+(\g)$ of $\g$ is the intersection of all kernels of all self-derivations of $\g$ that are diagonalizable with eigenvalues in $\mathbf{N}=\{0,1,\dots\}$. We say that $\g$ is essentially contractable if $\cni^+(\g)=\{0\}$.\end{defn}

Clearly $\cni(\g)\subset\cni^+(\g)$; this is not always an equality; for instance it can happen that $\cni(\g)=\{0\}$ but $\cni^+(\g)=\g$: this holds when $\g$ is semisimple, but also in the example of a nilpotent real Lie algebra described in \S\ref{exdere}.
 
Note that $\g$ is semi-contractable if and only if $\cni^+(\g)\neq\g$. We could make a similar definition allowing eigenvalues in $\mathbf{Z}$, but the resulting ideal (trapped between $\cni(\g)$ and $\cni^+(\g)$), which can be checked to be the intersection of kernels of $K$-diagonalizable self-derivations, would not behave well with respect to field extensions (e.g., in the example of \S\ref{exdere}). On the other hand, the positiveness ensures a good behavior:

\begin{thm}[see Th.\ \ref{stabex}]
Let $K$ be a $\mathbf{Q}$-field and $L$ an extension, and $\g$ a finite-dimensional Lie $K$-algebra. Then $\cni^+(L\ot_K\g)=L\ot_K\cni^+(\g)$. In particular, $\g$ is essentially contractable if and only $L\ot_K\g$ is essentially contractable.
\end{thm}

Note that essentially contractable does not imply contractable, since it does not even imply nilpotent: for example the non-nilpotent 2-dimensional Lie algebra, is essentially contractable. The implication does not even hold for nilpotent Lie algebras: the smallest counterexamples are 
7-dimensional and not even flexible (see Remark \ref{mag01}).

\subsubsection{Small dimension}\label{smalldim}Let us now put these results in light of the classification of small-dimensional nilpotent Lie algebras. \label{nlaaaa}

In the following table, we write a statement holding in low dimensions, and in the right column we write the largest dimension for which it holds. Let us begin by a few statements about NLAs (Nilpotent Lie Algebras) mainly related to fields of definition

\medskip

  \begin{tabular}{ccc}
statement & dim. & restriction\\
    \hline
NLAs with isomorphic complexifications are isomorphic & $\le 5$ & (${K^*}^2\neq K^*)$\\
every NLA is defined over $\mathbf{Q}$ & $\le 6$ & ($K=\mathbf{R},\mathbf{C}$)\\
there are finitely many isomorphism classes of NLAs & $\le 6$ & ($K=\mathbf{R},\mathbf{C}$) \\
    \hline
   \end{tabular} 

\medskip

The same holds if we restrict to Carnot Lie algebras, for instance the $\g_{7,3,1(i_\lambda)}$ in Magnin's classification \cite{Mag} form a 1-parameter family of 7-dimensional Carnot Lie algebras.

  \begin{tabular}{ccc}
statement & dim. \\
    \hline
every NLA is Carnot & $\le 4$ \\
every NLA is contractable & $\le 6$ \\
every nonzero NLA is semicontractable & $\le 6$ \\
every NLA is flexible & $\le 6$ \\
every NLA is either char.\ nilpotent or essentially contractable & $\le 7$ \\    \hline
   \end{tabular} 

\medskip

The largest dimension for which a classification (over an algebraically closed $\mathbf{Q}$-field) is available is 7. It includes 5 1-parameter families and is generally described as a list of about 154 types:
\begin{itemize}
\item 31 being decomposable as nontrivial direct product, thus contractable;
\item among the 123 indecomposable types
\begin{itemize}
\item 8 types, denoted $\g_{7,0,*}$ in Magnin's classification \cite{Mag}, including one 1-parameter family, consist of characteristically nilpotent NLAs;
\item 4 types, denoted $\g_{7,1,0*}$ in Magnin's classification \cite{Mag}, consist of NLAs that are not flexible, but are semicontractable and essentially flexible;
\item the other 111 types, including four 1-parameter families, are contractable; among them, 36 (including one 1-parameter family) are Carnot.
\end{itemize}
\item In dimension $\le 6$, the classification yields, one indecomposable NLA in dimension 3, 1 in dimension 4, 6 in dimension 5 (including 2 non-Carnot), and 20 in dimension 6 (including 10 non-Carnot). The two non-Carnot 5-dimensional NLAs can be described as follows:
\begin{itemize}\item The Lie algebra denoted $\g_{5,3}$ in Magnin's classification and $\mkl_{5,5}$ in de Graaf's classification, with basis $(X_i)_{1\le i\le 5}$ with nonzero brackets $[X_1,X_3]=X_4$, $[X_1,X_4]=X_5$ and $[X_2,X_3]=X_5$. Its nilpotency length is 3. It is not Carnot, for instance because its center is 1-dimensional but the center of the associated Carnot-graded Lie algebra is 2-dimensional. (Note that although $\mkl_{5,5}$ is indecomposable as a direct product, $\mathrm{Car}(\mkl_{5,5})\simeq\mkl_{5,3}$ splits as a direct product with the abelian factor generated by $X_2$.)
\item The Lie algebra $\g_{5,6}$ (or $\mkl_{5,6}$) with basis $(X_i)_{1\le i\le 5}$ with nonzero brackets $[X_1,X_i]=X_{i+1}$ ($i=2,3,4$) and $[X_2,X_3]=X_5$. Its nilpotency length is 4. Its associated Carnot Lie algebra $\g_{5,5}$ (or $\mkl_{5,7}$) is defined in the same way except $[X_2,X_3]=0$.
 \end{itemize}
\end{itemize}

\subsection{Systolic growth of nilpotent groups}\label{apl}

We now turn to the study of the systolic growth for finitely generated virtually nilpotent groups; as the systolic growth is invariant under passing to finite index subgroups, it is no restriction to focus to torsion-free finitely generated nilpotent groups. The main question raised in \S\ref{iisg} was to understand when the systolic growth and growth are equivalent.
This question is solved for nilpotent groups in the following theorem, which provides a geometric characterization of Carnot simply connected Lie groups (among those admitting lattices).

\begin{thm}[Theorem \ref{corlat2}]\label{corlat}
Let $G$ be a simply connected nilpotent real Lie group whose growth rate is polynomial of degree $\delta$, and whose Lie algebra $\g$ is definable over $\mathbf{Q}$. Equivalences:
\begin{enumerate}[(i)]
\item\label{carr} the Lie algebra $\g$ is Carnot (over $\mathbf{R}$) 
\item\label{elasy} every lattice in $G$ has systolic growth $\simeq n^\delta$
\item\label{slasy} some lattice in $G$ has systolic growth $\simeq n^\delta$
\item\label{ecovol} $G$ admits a sequence $(\Gamma_n)$ of lattices with systole $u_n\to\infty$ and covolume $\preceq u_n^\delta$.
\end{enumerate}
\end{thm}

Note that (\ref{slasy})$\Rightarrow$(\ref{ecovol}) is clear. The arithmeticity of lattices (see \cite{Ragh}) shows that definability over $\mathbf{Q}$ is equivalent to the existence of a lattice, whence (\ref{elasy})$\Rightarrow$(\ref{slasy}); more precisely, any lattice yields a $\mathbf{Q}$-structure on $\g$. Assuming (\ref{carr}), we use Theorem \ref{main2} (i.e., the Carnot case of Theorem \ref{beingca}) in order to show that some Carnot grading is defined over $\mathbf{Q}$, which allows to prove (\ref{elasy}). Finally the implication (\ref{ecovol})$\Rightarrow$(\ref{carr}) consists, roughly speaking, in rescaling $G$, and view some Gromov-Hausdorff limit $\Xi$ of the $\Gamma_n$ as a lattice in the asymptotic cone of $G$ and then observe that $\Gamma_n$ is isomorphic to $\Xi$ for $n$ large enough. This requires some preliminaries to ensure that $\Xi$ is indeed a lattice, and that $\Gamma_n$ converges to $\Xi$ in the space of marked groups.

The equivalence between (\ref{carr}) and (\ref{elasy}) was suggested by Gromov \cite[p.333]{Gro}, with, as only comment, the easy checking of (\ref{elasy}) in the case of the Heisenberg group. The proof of (\ref{carr})$\Rightarrow$(\ref{elasy}) is based on the same construction in general, but as we already mentioned, it makes, beforehand, a crucial use of Theorem \ref{main2} in its full generality (when the field extension is $\mathbf{Q}\subset\mathbf{R}$), and Gromov made no hint towards proving that any of the other properties implies~(\ref{carr}).

Using (\ref{ecovol}), any lattice in a non-Carnot simply connected nilpotent Lie group of polynomial growth of degree $\delta$ has systolic growth $\gg n^\delta$; it would be interesting to improve this estimate. For instance, for both non-Carnot 5-dimensional Lie algebras $\mk{l}_{5,5}$, $\mk{l}_{5,6}$ mentioned earlier (before \S\ref{apl}), we can check that the systolic growth is $\preceq n^{\delta+1}$ (with $\delta$ the degree of growth, 8 and 11 respectively) and I do not know if it is optimal in these cases. In general, the obvious upper bound $\sigma(n)\preceq n^{c\dim(G)}$, where $c$ is the nilpotency length, given by congruence subgroups is easy to improve (see Proposition \ref{uppersys}), but the precise behavior remains unclear and its study could shed light on how to quantify the lack of being Carnot.

\begin{rem}\label{rgirth2}
For finitely generated nilpotent groups, while obviously the residual girth $\sigma^\lhd$ is polynomially bounded (as we see using an embedding into upper unipotent integral matrices), it is in general asymptotically much larger that the systolic growth $\sigma$: for instance for the Heisenberg group we easily obtain that $\sigma^\lhd(n)\simeq n^6$ (while $\sigma(n)\simeq n^4$). 

The proof of Theorem \ref{corlat} actually shows that in the Carnot case, the uniform systolic growth (see Remarks \ref{rgirth} and \ref{rgirth2}) is asymptotically equivalent to the growth; I do not know whether it is asymptotically equivalent to the systolic growth for all finitely generated nilpotent groups.\end{rem}

\subsection{Cohopfian properties for nilpotent groups}
This part has a strong similarity with the previous one, since we characterize one property of finitely generated torsion-free nilpotent groups in terms of a property of Lie algebras that we have shown to be invariant under extensions of scalars. However, unlike in the case of systolic growth, we directly obtain the characterization of the cohopfian properties in terms of the rational Lie algebra. The invariance of the Lie algebra properties under extensions of scalars, nevertheless, appears as a way to recognize easily these properties, especially when we only have access to the complexification of the Lie algebra, as in most of the available classifications; they also show that the property does not differ when we consider two lattices in the same simply connected nilpotent Lie group.

The most familiar examples of infinite finitely generated torsion-free nilpotent groups fail to be cohopfian; however, Belegradek \cite{Bel} observed that there exists cohopfian examples: for instance, those for which the Malcev Lie algebra is characteristically nilpotent (i.e.\ has a virtually unipotent automorphism group). He gave a criterion for such a group $\Gamma$ to be non-cohopfian; his criterion \cite[Theorem 1]{Bel} is that the real Malcev completion $\g$ admits an automorphism mapping $\log(\Gamma)$ into itself, and of determinant of norm greater than~1. However this criterion depends on $\Gamma$ and from this characterization it is by no ways clear whether it changes when $\Gamma$ is replaced by a finite index subgroup. Actually, being cohopfian is not inherited by subgroups of finite index (see Appendix \ref{apfi}); however the following simple characterization shows that for finitely generated nilpotent groups, it is a commensurability invariant, and even only depends on the real Malcev completion.

\begin{thm}[Cor.\ \ref{corno}]\label{crcc}
Let $G$ be a simply connected nilpotent real Lie group whose Lie algebra $\g$ is definable over $\mathbf{Q}$. 
Equivalent statements:
\begin{enumerate}[(i)]
\item\label{nch1} $\g$ is semi-contractable;
\item\label{nch2} every lattice of $G$ is non-cohopfian;
\item\label{nch3} some lattice of $G$ is non-cohopfian.
\end{enumerate}
\end{thm}

Part of this theorem has independently been obtained by Dekimpe and Der\'e, namely, the statement that a finitely generated torsion-free nilpotent nilpotent Lie group is co-hopfian if and only if its {\em rational} Lie algebra is semi-contractable. This is one half of the the proof of Theorem \ref{corno}; the other (independent) part being that the rational Lie algebra is semi-contractable if and only if the real Lie algebra is semi-contractable, which is part of Theorem \ref{beingca}.

A naive expectation for proving (\ref{nch3})$\Rightarrow$(\ref{nch1}) would be to consider a non-injective endomorphism and extend it to the Malcev completion, hoping that the resulting endomorphism necessarily has all its eigenvalues of modulus $\ge 1$. This is not always the case, as the following simple example shows: $G=\mathbf{R}^2$, $\Gamma=\mathbf{Z}^2$ and the endomorphism given by the matrix $\begin{pmatrix}4 & 2\\ 2 & 2\end{pmatrix}$, whose eigenvalues are $3\pm\sqrt{5}$ (=0.76{\dots} and 5.23\dots).

A similar statement, with a similar proof, is the following:

\begin{thm}[Cor.\ \ref{corno}]\label{crcd}
Let $G$ be a simply connected nilpotent real Lie group whose Lie algebra $\g$ is definable over $\mathbf{Q}$. 
Equivalent statements:
\begin{enumerate}[(i)]
\item\label{dch1} $\g$ is contractable;
\item\label{dch2} every lattice of $G$ is dis-cohopfian;
\item\label{dch3} some lattice of $G$ is dis-cohopfian.
\end{enumerate}
\end{thm}

\begin{rem}
Fix a rational structure on $G$. When $\g$ is known to be contractable (resp.\ semi-contractable) over $\mathbf{Q}$, it is easy to check that some lattice in $G$ contained in $G_\mathbf{Q}$ is dis-cohopfian (resp.\ non-cohopfian). However to obtain the conclusion for {\em every} such lattice requires more work, mainly encapsulated in the technical Lemma \ref{defendo}. 
\end{rem}

\begin{thm}[Cor.\ \ref{corwdi}]\label{crcw}
Let $G$ be a simply connected nilpotent real Lie group whose Lie algebra $\g$ is definable over $\mathbf{Q}$. 
Equivalent statements:
\begin{enumerate}[(i)]
\item\label{wdch1} $\g$ is essentially contractable;
\item\label{wdch2} every lattice of $G$ is weakly dis-cohopfian;
\item\label{wdch3} some lattice of $G$ is weakly dis-cohopfian.
\end{enumerate}
\end{thm}

As an example of a corollary of these results, we have the following corollary. Recall that the {\em Hirsch length} of a finitely generated nilpotent group $\Gamma$ is the number of infinite subfactors in any composition series of $\Gamma$ with cyclic subfactors; it is also the dimension of any simply connected nilpotent Lie group admitting $\Gamma$ as a lattice.

\begin{cor}~
\begin{itemize}
\item
Every finitely generated torsion-free nilpotent group of Hirsch length $\le 6$ is dis-cohopfian.
\item If $\Gamma$ is a finitely generated torsion-free nilpotent group of Hirsch length 7, then either it is cohopfian or weakly dis-cohopfian (dis-cohopfian or not); this does not hold for Hirsch length 8.
\item Every finitely generated torsion-free 2-step nilpotent group is dis-cohopfian.
\end{itemize}
\end{cor}

\medskip \noindent {\bf Questions left open.}

This study of cohopfian properties was limited to torsion-free nilpotent groups in order to keep the size of the paper reasonable. It would be interesting to investigate finitely generated virtually nilpotent groups with no restriction.

\begin{que}
Let $\Gamma$ be a finitely generated virtually nilpotent group, and $\Gamma'$ a torsion-free nilpotent subgroup of finite index. Is it true that $\Gamma$ is cohopfian if and only if $\Gamma'$ is cohopfian?
\end{que}

The answer is yes for virtually abelian groups, by a non-trivial result of Delzant and Potyagailo \cite[Proof of Theorem~D]{DP}, namely every infinite finitely generated virtually abelian group is non-cohopfian.

If we turn to dis-cohopfian and weak dis-cohopfian properties, we need to introduce the polyfinite radical $W(\Gamma)$, which for an arbitrary group $\Gamma$ is the subgroup generated by all finite normal subgroups of $\Lambda$. For a virtually polycyclic group $\Gamma$, the subgroup $W(\Gamma)$ is finite; moreover it follows using \cite[Proposition 2.7]{CorInd} that every injective endomorphism of $\Gamma$ maps $W(\Gamma)$ onto itself. It particular, non-triviality of $W(\Gamma)$ is an obstruction for $\Gamma$ to be weakly dis-cohopfian (and dis-cohopfian).

\begin{que}
Let $\Gamma$ be a finitely generated virtually nilpotent group with $W(\Gamma)=\{1\}$, and $\Gamma'$ a torsion-free nilpotent subgroup of finite index. Is it true that $\Gamma$ is dis-cohopfian (resp.\ weakly dis-cohopfian) if and only if $\Gamma'$ is dis-cohopfian (resp.\ weakly cohopfian)?
\end{que}

In the virtually polycyclic case, finding a general characterization remains widely open. Many semidirect products of the form $\mathbf{Z}^d\rtimes\mathbf{Z}$ were shown to be weakly dis-cohopfian in \cite{NP}.

\begin{que}
Can the cohopfian property be characterized, for a virtually polycyclic group, in terms of Lie algebras? is it sensitive to passing to finite index subgroups? What if we restrict to those polycyclic groups that are Zariski-dense in a connected algebraic group?

What about the weak dis-cohopfian property, if we assume in addition that the polyfinite radical is trivial?
\end{que}

For the dis-cohopfian property, we have in mind that a contractable Lie algebra is always nilpotent. This suggests the following:

\begin{que}
Is it true that no virtually polycyclic group of exponential growth is dis-cohopfian?
\end{que}

\medskip \noindent {\bf Organization of the paper.} 

Many properties of the Lie algebras we are interested in, such as the existence of gradings with given properties, only depend on their group of automorphisms. Therefore it is convenient to study these properties by forgetting the Lie algebra structure and retaining the group of automorphisms, and more generally considering a Zariski-closed subgroup $G$ in $\GL(V)$ and study gradings on $V$ defined by sub-tori in $G$. This is done in the preliminary Section \ref{graza}. Next, in Section \ref{graal}, we specify to group automorphisms of algebras; actually all results in this section were initially written for Lie algebras, but none of the Lie algebras axioms are used, except the existence of a bilinear law; therefore this section is written for arbitrary algebras. This is not the most general context but all applications we have in mind concern Lie algebras and further generalizations (e.g., considering several laws, or ternary laws, etc.)\ would make the text harder to read. In addition, for the (counter)examples given throughout the text, and in Section \ref{secount} for some more consistent ones, we especially focus on Lie algebras.

In Section \ref{conig}, we prove the theorems on cohopfian properties; this makes use of the contractive decomposition which is studied in Sections \ref{graza} and \ref{graal}.

Section \ref{gelani} is mostly independent of the others; it starts with a notion of systolic growth for locally compact groups; it includes the proof of Theorem \ref{corlat} on the systolic growth, or rather the more general Theorem \ref{corlat2}. Its main bulk is the geometric part of the proof, namely (\ref{ecovol})$\Rightarrow$(\ref{carr}).

\medskip \noindent {\bf Acknowledgements.} I thank Goulnara Arzhantseva and Yves Benoist for useful conversations, and Pierre de la Harpe and Khalid Bou-Rabee for a number of corrections and suggestions. I thank Igor Belegradek for letting me know about the work of Karel Dekimpe and Jonas Der\'e. I thank Jonas\ Der\'e for pointing a mistake in a previous version of this work.

\setcounter{tocdepth}{1}
\tableofcontents


\section{Gradings associated to a Zariski closed subgroup of $\GL_d$}\label{graza}

\subsection{Linear algebraic $K$-groups}

In this section, we let $K$ be a $\mathbf{Q}$-field, i.e., a field of characteristic zero. The facts we will use about linear algebraic $K$-groups $G$ (whose unit component is denoted by $G^\circ$) are the following.

\begin{itemize}
\item $G$ admits a maximal torus that is defined over $K$ (such tori are not necessarily conjugate over $K$); more generally, every $K$-defined torus in $G$ is contained in a maximal torus that is defined over $K$ (see \cite[A.1.2]{Con}).
\item \cite[\S 8.2]{BT} $G$ admits maximal $K$-split tori, which are all conjugate under $G^\circ(K)$ (but these are not necessarily maximal tori). Their common dimension $r$ is called the $K$-rank of $G$. 
\item \cite[\S 5.1]{BSer} $G$ admits a Levi factor $R$ defined over $K$, so that $G=U\rtimes R$ with $U$ the unipotent radical; all such $R$ are conjugate under $U(K)$; every reductive $K$-subgroup of $G$ (not necessarily connected) is contained in a $K$-defined Levi factor.
\end{itemize}

\begin{lem}\label{pairsemb}
Let $G'\subset G\subset\GL(V)$ be closed subgroups. Then all pairs $(T',T)$, where $T'$ is a maximal torus in $G'$ and $T$ a maximal torus in $G$ containing $T'$, are conjugate under $G^\circ$.
\end{lem}
\begin{proof}
Let $(T'_1,T_1)$ and $(T'_2,T_2)$ be two such pairs. Then there exists $g\in (G')^\circ$ such that $gT'_1g^{-1}=T'_2$. Define $T_3=gT_1g^{-1}$. Let $N$ be the normalizer of $T'_2$ in $G$. Then $T_2$ and $T_3$ are maximal tori of $N$ containing $T'_2$. Hence there exists $h\in N^\circ$ such that $hT_3h^{-1}=T_2$. Since $hT'_2h^{-1}=T'_2$, it follows that $hg(T'_1,T_1)(hg)^{-1}=(T'_2,T_2)$.
\end{proof}

\subsection{Maximal split tori and gradings}\label{kcg}
 
Let $V$ be the affine $n$-space, i.e., $V(K)=K^n$ for every field $K$. Let $T\subset\GL_d=\GL(V)$ be a $K$-split torus, say $r$-dimensional. Then $T$ defines a grading of $V$ in the group $\sfX(T)\simeq\mathbf{Z}^r$ of multiplicative characters of $T$. Here $V_\chi=\{v\in V\mid \forall t\in T:tv=\chi(t)v\}$. Note that $T$ being $K$-split means that all $\chi\in\sfX(T)$ are $K$-defined.

Let $G\subset\GL(V)$ be a $K$-closed subgroup, and $r$ its $K$-rank. Every maximal $K$-split torus in $G$ thus yields a grading of $V$ in $\mathbf{Z}^r$. Moreover, any two such gradings (for two choices of maximal split tori and identification of their group of characters with $\mathbf{Z}^r$) are conjugate under $G^\circ(K)$ and $\GL_d(\mathbf{Z})$, in the sense that if $(\g_n)_{n\in\mathbf{Z}^r}$ and $(\g'_n)_{n\in\mathbf{Z}^r}$ are two such gradings, then there exists $s\in G^\circ(K)$ and $f\in\GL_r(\mathbf{Z})$ such that $\g'_n=s(\g_{f(n)})$ for all $n\in\mathbf{Z}^r$.

\subsection{Positive weights, contractive decomposition and fine tori}

(The forthcoming notions do not depend on a field of definition.)

If $T\subset\GL(V)$ is a torus, it defines a grading $V=\sum_{\alpha\in\sfX(T)}V_\alpha$. We say that $\alpha\in\sfX(T)$ is a {\em weight} if $V_\alpha\neq 0$. We say that a homomorphism $\sfX(T)\to\mathbf{R}$ is {\em non-negative} if it maps weights to non-negative numbers. We say that $\alpha\in\sfX(T)$ is {\em positive} (or $T$-{\em positive} if we want to emphasize $T$) if there exists a non-negative homomorphism $f$ such that $f(\alpha)>0$ (such a homomorphism can then be chosen to be valued in $\mathbf{Z}$).

\begin{defn}\label{codec}
The {\em contractive decomposition} associated to $T$ is the decomposition $V=V_{[0]}\oplus V_{[+]}$, where $V_{[+]}$ is the sum of all $V_\alpha$ when $\alpha$ ranges over positive weights, and $V_{[0]}$ is the sum of $V_\alpha$ when $\alpha$ ranges over non-positive weights. We write $V_{[0]}^T$ and $V_{[+]}^T$ if we want to emphasize $T$.

The dimensions of $V_{[+]}$ and $V_{[0]}$ are called the {\em contracted dimension} and {\em uncontracted dimension} of $(V,T)$, or of $(V,G)$ whenever $G$ admits $T$ as a maximal torus.
\end{defn}

If $T'\subset T$ is a subtorus, we have $V_{[0]}^{T'}\supset V_{[0]}^{T}$ and $V_{[+]}^{T'}\subset V_{[+]}^{T}$.

\begin{defn}
Given $T'\subset T\subset\GL(V)$, the torus $T'$ is {\em fine in $T$} if the restriction map $\sfX(T')\to\sfX(T)$ maps positive $T$-weights to positive $T'$-weights, or equivalently if $V_{[+]}^{T'}=V_{[+]}^T$, or equivalently if $V_{[0]}^{T'}=V_{[0]}^T$.
\end{defn}

If $G,G'\subset\GL(V)$ are closed subgroups, we say that $G'\subset G$ is fine if, denoting by $T'$ some maximal torus in $G'$ and $T$ some maximal torus of $G$ containing $T'$, we have $T'$ fine in $T$ (this does not depend on $T,T'$). 

The following lemma is immediate:

\begin{lem}
Given $G''\subset G'\subset G$, the subgroup $G''$ is fine in $G$ if and only if $G''$ is fine in $G'$ and  $G'$ is fine in $G$.\qed
\end{lem}

\begin{defn}
Let $T\subset\GL(V)$ be a torus. A {\em $T$-fine cocharacter} is a cocharacter $\GL_1\to T$ whose associated grading is an $\mathbf{N}$-grading and satisfies $V_0=V_{[0]}^T$. Every grading obtained this way is called a {\em fine $\mathbf{N}$-grading} for $T$. If $G\subset\GL(V)$ is an arbitrary closed subgroup, a {\em $G$-fine cocharacter} is a fine cocharacter of some maximal torus in $G$.
\end{defn}

It follows from the definition that the image of every $T$-fine cocharacter is fine in $T$. There always exist $T$-fine cocharacters; more precisely:

\begin{lem}\label{excom}
Every torus $T$ admits a $T$-fine cocharacter; 
\end{lem}
\begin{proof}
For each positive weight $\alpha$, choose a non-negative homomorphism $f_\alpha:\sfX(T)\to\mathbf{Z}$ such that $f_\alpha(\alpha)>0$. Then $f=\sum_\alpha f_\alpha$ is a non-negative homomorphism that is positive on all positive weights. This homomorphism is induced by some cocharacter, and it satisfies the required properties.
\end{proof}

\begin{thm}\label{pofi}
Let $G\subset\GL(V)$ be a $K$-closed subgroup. 
Let $R$ be a $K$-defined reductive Levi factor; let $D$ be the maximal $K$-split torus in the center $Z(R)$ of $R$. Then $D\subset G$ is fine.
\end{thm}

\begin{proof}
First assume that $G$ is reductive and connected. Write $G=ST$, where $S$ is the semisimple part and $T=Z(G)^\circ$ the maximal normal torus in $G$. Write $T=AD$, where $A$ is the maximal $K$-anisotropic torus and $D$ the maximal $K$-split torus. Write $M=SA$, so that $G=MD$. Decompose $V$ as a direct sum $\bigoplus_iV_i$ of $K$-irreducible components with respect to the action of $G$. On each $V_i$, the action of $M$ has determinant 1, because $M$ admits no nontrivial $K$-defined multiplicative character, and the action of $D$ is scalar, given by a weight $\alpha_i\in\sfX(D)$. Fix a maximal torus $L$ in $M$ (possibly not $K$-defined). Write the weights (counted with multiplicity) of $L\times D$ in $V_i$ as $(\beta,\alpha_i)\in\sfX(L)\times\sfX(D)$, where $\beta$ ranges over a set $P_i$ of weights of $L$. Since the action of $L$ has determinant 1, we have $\sum_{\beta\in P_i}\beta=0$. Let $(\beta_0,\alpha_j)$ be a positive weight. This means that there exists a homomorphism $(f,g)$ on $\sfX(L\times D)$, valued in $\mathbf{Z}$, mapping $(\beta_0,\alpha_j)$ to a positive number and all weights $(\beta,\alpha_i)$ to non-negative numbers. The latter reads: for all $i$ and every $\beta\in P_i$, we have $f(\beta)+g(\alpha_i)\ge 0$, and $f(\beta_0)+g(\alpha_j)>0$. Summing over $\beta$, we obtain $\#(P_i)g(\alpha_i)\ge 0$, and hence $g(\alpha_i)\ge 0$.  Also for $i=j$ this provides $g(\alpha_j)>0$. This shows that $g$ is non-negative on weights and positive on $\alpha_j$. This proves that whenever $(\beta_0,\alpha_j)$ is a positive weight of $LD$, then $\alpha_j$ is a positive weight of $D$. This proves that $D\subset LD$ is fine; since $LD$ is a maximal torus in $G$, it is fine in $G$, and hence it follows that $D$ is fine in $G$. 

Now assume that $G$ is reductive, but not necessarily connected. Define $S,T,A,M$ from $G^\circ$ as above, and let $D'$ be the maximal $K$-split torus in $Z(G^\circ)$. Then $G^\circ=MD'$, and $D\subset D'$. By the connected case, $D'$ is fine in $G^\circ$, and hence in $G$. So we have to show that $D$ is fine in $D'$. By Lemma \ref{excom} (or its proof) there exists a homomorphism $f:\sfX(D')\to\mathbf{Z}$ mapping weights to non-negative integers and mapping positive weights to positive integers. The finite group $G/G^\circ$ acts on $D'$ by conjugation, inducing an action on $\sfX(D')$. This actions permutes the weights and permutes the positive weights. Therefore if we average $f$ by defining $\bar{f}=\sum_{s\in G/G^\circ}s\cdot f$, then $\bar{f}$ is a homomorphism $\sfX(D)\to\mathbf{Z}$ that is non-negative on weights, positive on positive weights, and $G$-invariant. Therefore the corresponding cocharacter is fine in $D'$ and has its image contained in the center of $G$, and hence in $D$. This proves that $D$ is fine in $D'$, and hence in $G$.
 
Finally, if $G$ is arbitrary, then by the reductive case $D$ is fine in $R$, and it is always true (and obvious) that $R$ is fine in $G$. Hence $D$ is fine in $G$.
\end{proof}

\begin{cor}\label{finespli}
If $G$ is $K$-defined, every maximal $K$-split torus is fine in $G$.
\end{cor}
\begin{proof}
Let $R$ and $D$ be given as in Theorem \ref{pofi}. Then $D$ is fine in $G$ by the theorem. Let $D'$ be a maximal $K$-split torus of $G$ containing $D$. Then $D'\subset G$ is fine as well.
\end{proof}

\begin{cor}
If $G\subset\GL(V)$ is $K$-defined and $T$ is a maximal $K$-split torus, then the (un)contractive dimension of $(V,G)$ equals the (un)contractive dimension of $(V,T)$.
\end{cor}

Note that the latter corollary is trivial when $T$ is a maximal torus.

A restatement of Corollary \ref{finespli} is the following:

\begin{cor}\label{tfi}
If $T\subset\GL(V)$ is a $K$-defined torus, then there is a $K$-defined $T$-fine cocharacter.
\end{cor}
\begin{proof}
Let $T'$ be the maximal $K$-split torus in $T$. By Lemma \ref{excom}, $T'$ admits a $K$-defined fine cocharacter $\sigma$. By Corollary \ref{finespli}, $T'$ is fine in $T$. Hence $\sigma$ is also fine as a cocharacter of $T$.
\end{proof}

Another application of Theorem \ref{pofi} concerns finite subgroups.

\begin{cor}\label{splifi}
Let $S$ be a reductive $K$-subgroup of $G$ (possibly not connected, e.g., finite). Then every maximal $K$-split torus in the centralizer of $S$ is fine in $G$. Equivalently, there exists a $K$-defined fine cocharacter whose image centralizes $S$.
\end{cor}
\begin{proof}
Let $R$ be a $K$-defined Levi factor containing $S$. By Theorem \ref{pofi}, the maximal $K$-split torus $D$ in $Z(R)$ is fine in $G$. Let $C$ be the centralizer of $S$ in $G$. Then $D\subset C$. Let $D'$ be a maximal $K$-split torus in $C$ containing $D$. Since $D$ is fine in $G$, so is $D'$.
\end{proof}

Let us also provide an explicit corollary in extremes cases of the contractive decomposition. Define a cocharacter $\GL_1\to\GL(V)$ to be {\em non-negative} if it defines a grading in $\mathbf{N}$, and {\em positive} if moreover 0 is not a weight.

\begin{cor}\label{Sreduc}
Let $G\subset\GL(V)$ be a $K$-closed subgroup, and $S$ a $K$-closed reductive subgroup.
\begin{itemize}
\item If $G$ has a positive cocharacter, then it admits a positive $K$-defined cocharacter valued in the centralizer of $S$;
\item if $G$ has a non-trivial non-negative cocharacter, then it admits a $K$-defined non-trivial non-negative cocharacter valued in the centralizer of $S$. 
\end{itemize}
\end{cor}
Note that for a cocharacter, the condition that it is valued in the centralizer of $S$ implies that the corresponding grading is preserved by $S$.

\section{Generalities, grading and extensions of scalars}\label{graal}

\subsection{Arbitrary algebras}\label{arba}

Let $R$ be a scalar ring, that is, an associative, unital, commutative ring. In this section, we call {\em $R$-algebra} an $R$-module endowed with an arbitrary $R$-bilinear law, with no further assumption; it in particular includes Lie algebras, associative non-unital algebras and various generalizations. 

If $\g$ is an $R$-algebra, we define its {\em lower series} as follows: $\g^{(1)}=\g$, and, for $k\ge 2$,
\[\g^{(k)}=\sum_{i,j\ge 1,i+j=k}\g^{(i)}\g^{(j)}.\]

This is a bilateral ideal. For $k\ge 0$, if $\g^{(k+1)}=0$, then $\g$ is called {\em $k$-step nilpotent}; if this holds for some $k$, then $\g$ is called nilpotent; its {\em nilpotency length} is the smallest $k$ for which this holds.
When $\g$ is a Lie algebra, it is the usual lower central series.

If $(A,+)$ is a magma (a set endowed with an arbitrary binary law), a grading of an algebra $\g$ in $A$ (or $A$-grading) is a direct sum decomposition $\g=\bigoplus_{\alpha\in A}\g_\alpha$ of $\g$ as an $R$-module, such that $\g_\alpha\g_\beta\subset\g_{\alpha+\beta}$. The weights of $\g$ are those $\alpha\in A$ such that $\g_\alpha\neq\{0\}$.

If $V$ is an arbitrary $A$-graded $R$-module, then any graded subquotient of $V$ inherits an $A$-grading. In particular, if $\g$ is graded in $A$, then $\g/\g^{(2)}$ inherits a grading in $A$; its weights are called {\em principal weights} of $\g$. If $\g$ is nilpotent, then it can be checked that the set of weights is contained in the submagma generated by principal weights.

We can extend to arbitrary algebras the definitions of Lie algebras. 

\begin{defn}
Let $K$ be a $\mathbf{Q}$-field and let $\g$ be a $K$-algebra that is finite-dimensional as vector space.
\begin{itemize}
\item $\cni^+(\g)$ is the intersection of kernels of $K$-diagonalizable self-derivations of $\g$ with all eigenvalues in the non-negative integers $\mathbf{N}$;
\item $\cni(\g)$ is the intersection of kernels of semisimple self-derivations of $\g$;
\item $\g$ is contractable if it admits an algebra grading in positive integers, or equivalently if it admits a self-derivation with only positive integers as eigenvalues;
\item $\g$ is semi-contractable if $\cni^+(\g)\neq\g$, or equivalently if $\g$ admits a nontrivial algebra grading in non-negative integers, or still equivalently if it admits a self-derivation with only non-negative integers as eigenvalues, and at least one positive eigenvalue;
\item $\g$ is characteristically nilpotent if $\cni(\g)=\g$;
\item $\g$ is flexible if it admits an invertible self-derivation;
\item $\g$ is essentially flexible if $\cni(\g)=\{0\}$;
\item $\g$ is essentially contractable if $\cni^+(\g)=\{0\}$.
\end{itemize}
\end{defn}

\begin{rem}In the realm of Lie algebras, characteristically nilpotent is a classical notion and terminology. Contractable is a classical notion but has no common terminology (it is sometimes called ``graded", but this is also used to mean ``Carnot" which is a strictly stronger notion). Even for Lie algebras, essentially flexible and essentially contractable are new notions; the latter is motivated by Theorem \ref{crcw}.
\end{rem}

\subsection{Carnot algebras}

Here the algebras are over an arbitrary commutative associative unital ring.

Carnot Lie algebras and associated Carnot-graded Lie algebras are important objects, appearing in different places (possibly first in Leger's paper \cite{Le}) under many more names (``graded", ``naturally graded", ``homogeneous", ``quasi-cyclic"), which are often inconvenient and ambiguous; the non-ambiguous ``fundamental graded" is also used by some authors (with negative gradings). The use of the word ``Carnot" in this context is common in sub-Riemannian and conformal geometry. Here we introduce them in the context of arbitrary algebras.

\begin{defn}
A {\em Carnot grading} on an algebra $\g$ is a algebra grading of $\g$ in $\mathbf{N}$ such that $\g$ is generated by $\g_1$ as an algebra. An algebra is {\em Carnot-graded} if it endowed with a Carnot grading, and {\em Carnot} if it admits one Carnot grading.
\end{defn}

Let $\g$ be an algebra, and let $(\g^{(i)})_{i\ge 1}$ be its lower series. The product maps $\g^{(i)}\times\g^{(j)}$ into $\g^{(i+j)}$ for all $i,j$; in particular, denoting $\mk{v}_i=\g^{(i)}/\g^{(i+1)}$, it induces a bilinear map $\mk{v}_i\times\mk{v}_j\to\mk{v}_{i+j}$ for all $i,j$.

\begin{defn}
The direct sum $\Car(\g)=\bigoplus_{i\ge 1}\mk{v}_i$, endowed with the bilinear law, is a graded algebra, called the {\em associated Carnot-graded}, or {\em associated graded} algebra to $\g$.
\end{defn}

The graded algebra $\Car(\g)$ is indeed Carnot-graded: this follows from the surjectivity of $\bigoplus_{i+j=n}\mk{v}_i\otimes\mk{v}_j\to \mk{v}_n$ for all $n\ge 2$, which implies by induction that $\g_n$ is contained in the subalgebra generated by $\g_1$ for all $n\ge 1$.

\begin{prop}\label{crac}
An algebra $\g$ is Carnot if and only if it is isomorphic (as an algebra) to its associated Carnot-graded algebra $\Car(\g)$. Moreover if these conditions hold, then
\begin{itemize}
\item for any Carnot grading on $\g$, the graded algebras $\g$ and $\Car(\g)$ are isomorphic;
\item for any two Carnot gradings on $\g$, there is a unique automorphism mapping the first to the second, and inducing the identity modulo $\g^{(2)}$.
\end{itemize}
\end{prop}
\begin{proof}
If $\g$ is isomorphic to $\Car(\g)$, then it admits a Carnot grading (inherited from $\Car(\g)$). 

For the second sentence, we observe that any Carnot-grading on $\g$ defines, in a canonical way, an isomorphism from $\g$ to $\Car(\g)$, which in restriction to $\g_i$ is the restriction of the projection $p_i:\g^{(i)}\to\g^{(i)}/\g^{(i+1)}$. This proves the first item. For the second item, suppose that we have two Carnot gradings $(\g_i)$ and $(\g'_i)$, defining as above isomorphisms $f,g:\g\to\Car(\g)$. Then for every $x\in\g_1,x'\in\g'_1$ we have $f(x)=p_1(x)$ and $g(x')=p_1(x')$. Hence, if we define $h=g^{-1}f$, then $h\in\Aut(\g)$ and $p_1(h(x))=g(h(x))=f(x)=p_1(x)$. This shows that $h$ equals the identity modulo $[\g,\g]$. Then $h$ maps the first grading to the second grading. The uniqueness is clear.
\end{proof}

\begin{cor}\label{autautgr}
Let $\g$ be a Carnot-graded algebra over $R$. Denote by $\Aut(\g)$ its automorphism group as an algebra and $\Aut(\g)_0$ its automorphism group as graded algebra. Let $\Aut(\g)_{\ge 1}$ be the group of automorphisms of the Lie algebra $\g$ inducing the identity on $\g/\g^{(2)}$. Then $\Aut(\g)_{\ge 1}$ is a normal subgroup and
\[\Aut(\g)=\Aut(\g)_0\ltimes\Aut(\g)_{\ge 1}.\]
\end{cor}
\begin{proof}
The group $\Aut(\g)_{\ge 1}$ is clearly normal.

If we have an automorphism $\phi$ of $\g$, then it maps the given Carnot grading to another one. By the last assertion of Proposition \ref{crac}, there exists $u\in\Aut(\g)_{\ge 1}$ 
mapping the first Carnot grading to the second one. Hence $v=u^{-1}\phi\in\Aut(\g)_0$, and $\phi=uv$. 
\end{proof}

If we work in a certain category of algebras $\g$ (e.g., associative, Lie,\dots) it is useful to know that $\Car(\g)$ is also in the same category. The following proposition proves it in many cases.

If $P=P(X_1,\dots,X_k)$ is any non-associative polynomial (with scalars in $R$) in the formal variables $X_1,\dots,X_k$, we say that $P$ is an {\em identity} for an algebra $\g$ if $P(x_1,\dots,x_k)=0$ for all $x_i\in\g$.

\begin{prop}
Let $P$ be an identity for $\g$. If $P$ is either multilinear or $P(X)=X^2$, then it is an identity for $\Car(\g)$.
\end{prop}
\begin{proof}
We need to show that $P(x_1,\dots,x_k)=0$ for all $x_1,\dots,x_k\in\Car(\g)$. By multilinearity, we can assume that each $x_i$ is homogeneous, say of degree $n_i$. Let $y_i$ be a lift of $x_i$ in $\g^{(n_i)}$. For $\pi_j$ be the projection from $\g^{(j)}$ to $\g^{(j)}/\g^{(j+1)}$ and let $n=\sum n_i$. Then for every $k$-multilinear polynomial $Q$ and all $z_i\in\g^{(n_i)}$, we have $\pi_n(Q(z_1,\dots,z_k))=Q(\pi_{n_1}(z_1),\dots,\pi_{n_k}(z_k))$. In particular, we deduce that $P(x_1,\dots,x_k)=0$.

For the case of $P(X)=X^2$, consider an element $\sum_{j\ge 1} x_j$ with $x_j\in\g^{(j)}/\g^{(j+1)}$ (almost all zero) and let $y_j$ be a lift of $x_j$ in $\g^{(j)}$ (chosen to be 0 if $x_j=0$). Then
\begin{align*}
\left(\sum x_j\right)^2= & \sum x_j^2+\sum_{k<\ell}(x_kx_\ell+x_\ell x_k)\\
= & \sum_j \pi_{2j}(y_j^2)+\sum_{k<\ell}\pi_{k+\ell}(y_ky_\ell+y_\ell y_k)\\
= & \sum_j \pi_{2j}(y_j^2)+\sum_{k<\ell}\pi_{k+\ell}((y_k+y_\ell)^2-y_k^2-y_{\ell}^2)=0,\qedhere
\end{align*}
because $y_j^2=(y_k+y_\ell)^2=0$ for all $j,k,\ell$.
\end{proof}

\begin{rem}
To be multilinear is a useful property for an identity: indeed, it implies that it passes to extensions of scalars. See the discussion in \cite{Sch}.
\end{rem}

\begin{ex}
Here are some further types of algebras defined by a multilinear identity:
\begin{itemize}
\item Associative algebras: $A(x,y,z)=(xy)z-x(yz)$ ($A$ is known as the {\em associator});
\item Pre-Lie algebras: $PL(x,y,z)=(xy)z-x(yz)-(xz)y+x(zy)=A(x,y,z)-A(x,z,y)$;
\item Novikov algebras: $PL$ and $N(x,y,z)=(xy)z-(xz)y$;
\item Leibniz algebras: $L(x,y,z)=(xy)z-x(yz)-(xz)y$.
\end{itemize}
Some more are defined by an identity with a quadratic variable, which can, by the obvious polarization formulas, be converted to a multilinear identity when 2 is invertible:
\begin{itemize}
\item alternative algebras: $AL_1(x,y)=A(x,x,y)$, $AL_2(x,y)=A(y,x,x)$, where $A$ is the associator;
\item Malcev algebras: $C'(x,y)=xy+yx$ and $M(x,y,z)=(xy)(xz)-((xy)z)x-((yz)x)x-((zx)x)y$.
\end{itemize}
Another classical notion is that of Jordan algebra; classically it is defined as an algebra satisfying the identity $C(x,y)=xy-yx$ (i.e., commutative) and the identity $J\!o(x,y)=(xy)(xx)-x(y(xx))=A(x,y,xx)$; if 6 is invertible, a simple verification is that this is equivalent to satisfy the multilinear identities $C$ and $J\!o'(x,y,z,w)=A(x,y,zw)+A(z,y,wx)+A(w,y,xz)$.
\end{ex}

\begin{lem}\label{carcarn2}
Let $K$ be a field. Let $\g$ be a nilpotent algebra over $K$; if $K$ has positive characteristic $p$, assume that its nilpotency length $c$ is at most $p+1$. Then $\g$ is Carnot over $K$ if and only if $\g$ has a self-derivation inducing the identity on $\g/\g^{(2)}$.
\end{lem}
\begin{proof}
The ``only if" part is clear, since any Carnot grading defines such a derivation, defined to be the multiplication by $i$ on $\g_i$. Conversely, assume that there exists such a self-derivation $\delta$ (and keep in mind that we do not assume $\g$ to be finite-dimensional). 

For any integer $i\ge 1$, define $\g_{[i]}$ as the eigenspace of $\delta$ for the eigenvalue $i$. By assumption, we have $\delta(x)-x\in\g^{(2)}$ for all $x\in\g$. By induction, we deduce that $\delta(x)-ix\in\g^{(i+1)}$ for all $i\ge 1$ and $x\in\g^{(i)}$. Then by a descending induction on $1\le i\le c$, we see that $(\delta-i)\dots (\delta-c)$ vanishes on $\g^{(i)}$. Eventually for $i=1$, this means that $\Delta=(\delta-1)\dots (\delta-c)$ vanishes on $\g$. 

First assume that $K$ has characteristic zero, or characteristic $p\ge c$. Then $1,\dots,c$ are distinct in $K$. This implies that the eigenspaces $\g_{[i]}$ for $1\le i\le c$ span their direct sum, and the vanishing of $\Delta$ implies that they actually span $\g$. Thus defining $\g_i=\g_{[i]}$ yields a Carnot grading.

Now assume the remaining case, namely $p=c-1$. Observe that $0=\Delta$ can be rewritten as $(\delta-1)^2(\delta-2)\dots (\delta-p-1)$. If we define $\g_{\langle 1\rangle}=\g_{\langle p\rangle}$ as the kernel of $(\delta-1)^2$ and $\g_{\langle i\rangle}=\g_{[i]}$ for $2\le i\le c-1$, then the $\g_{\langle i\rangle}$ for $1\le i\le p-1$ generate their direct sum, which by the vanishing of $\Delta$ equals $\g$, and $\g_{\langle i\rangle}\g_{\langle j\rangle}\subset\g_{\langle i+j\rangle}$ for all integers $i,j\ge 1$. Then define $\g_1$ as a supplement subspace of $\g^{(p)}$ in $\g_{\langle 1\rangle}$, define $\g_p=\g^{(p)}$, and $\g_i=\g_{\langle i\rangle}$ for $2\le i\le p-1$. Then $\g=\bigoplus_{i=1}^p\g_i$ and $\g^{(2)}=\bigoplus_{i=2}^p\g_i$, and, noting that $\g_p=\g_{\langle p\rangle}\cap\g^{(2)}$, we see that $(\g_i)$ is an algebra grading; thus it is a Carnot grading.
\end{proof}

\begin{rem}
In characteristic $p$, the nilpotency length condition in Lemma \ref{carcarn2} is optimal: consider the Lie algebra $\g$ with basis $(U_1,X_1,\dots,X_{p+2},Y_1,Z_1)$, with nonzero brackets $[U_1,X_i]=X_{i+1}$ for $1\le i\le p+1$ and $[Y_1,Z_1]=X_{p+2}$. It is easily checked not to be Carnot. On the other hand, it admits a grading in $\mathbf{Z}/p\mathbf{Z}$ where $\g_1$ has basis $(U_1,X_1,X_{p+1},Y_1,Z_1)$, where $\g_2$ has basis $(X_2,X_{p+2})$ and $\g_i$ has basis $(X_i)$ for $3\le i\le p$. This grading induces a derivation which is the identity on the abelianization (which has the basis $(U_1,X_1,Y_1,Z_1)$). On the other hand, we do not know whether the assumptions in positive characteristic can be relaxed for Theorem \ref{main2}.
\end{rem}

\begin{ex}
Every 2-step nilpotent algebra $\g$ is Carnot, as any supplement subspace of $\g^{(2)}$ yields a Carnot grading. However, the 3-step nilpotent 5-dimensional Lie algebra $\g_{5,3}$ from \S\ref{smalldim} is not Carnot. 
\end{ex}

\begin{ex}
Every nilpotent algebra of dimension at most 3 over a field is Carnot, by an easy verification. While all 4-dimensional nilpotent Lie algebras, over an arbitrary field, are Carnot, the following 4-dimensional commutative associative nilpotent algebra, defined over any ground field, is not Carnot: The algebra $\g$ defined with basis $(x,y,z,w)$ and nonzero products $x^2=z$, $xz=zx=w$, $y^2=w$ (alternatively described $K[x,y]_+/(xy,x^3-y^2)$, where $K[x,y]_+=(x,y)K[x,y]$ is the free associative (non-unital) commutative ring on 2 generators, or also described as the unique maximal ideal in $K[x,y]/(xy,x^3-y^2)$). The next example also provides a non-Carnot 4-dimensional Leibniz algebra.
\end{ex}

\begin{ex}
Here is a 4-dimensional nilpotent Leibniz algebra in which every self-derivation is nilpotent: it has the basis $(x_i)_{1\le i\le 4}$ with nonzero products $x_1^2=x_3$, $x_2x_1=x_3$, $x_2^2=x_4$, $x_3x_1=x_4$. (It is denoted $\mathfrak{R}_4(0)$ in \cite[Theorem 3.2]{AOR}.) Indeed, calling it $\g$, a computation shows that every derivation maps $\g$ into $\g^{(2)}$.
\end{ex}


\begin{thm}\label{main2}
Let $K\subset K'$ be fields of characteristic zero. Let $\g$ be a finite-dimensional algebra over $K$. Then $\g$ is Carnot over $K$ if (and only if) $K'\ot_K\g$ is Carnot over $K'$. More generally, this holds in positive characteristic $p$ when either
\begin{itemize}
\item the nilpotency length of $\g$ is at most $p+1$, or
\item $K$ is an infinite and perfect field.
\end{itemize}
\end{thm}

\begin{proof} We use the criterion of Lemma \ref{carcarn2}.
It is convenient to take a slightly schematic point of view and write $\underline{\g}$ for the functor $L\mapsto \underline{\g}_L=L\ot_K\g$ for every field extension $K\subset L$.

Let $\underline{D}$ be the affine space of self-derivations of $\underline{\g}$ that induce the identity on $\underline{\g}/\underline{\g^{(2)}}$: that is, $\underline{D}_L$ is the (possibly empty) affine space of self-derivations $\delta$ of $\underline{\g}_L$ such that $\delta(x)-x$ belongs to $\underline{\g^{(2)}}_L$ for all $x\in\underline{\g}_L$. Then $\underline{D}$ is a $K$-defined affine subspace of the space of linear endomorphisms of $\underline{\g}$. Then by assumption $\underline{D}_{K'}$ is non-empty. It follows that $\underline{D}_K$ is non-empty as well. In other words, $\g$ admits a self-derivation inducing the identity on $\g/\g^{(2)}$.

Now assume that $K$ is infinite perfect with no further restriction, and that $\g$ has nilpotency length $\ell\ge 1$ and $\g\otimes_K K'$ is Carnot over $K'$. We can assume $\g\neq\{0\}$.
Let $H$ be the algebraic subgroup of $\mathrm{GL}(\g)$ consisting of those automorphisms $h$ inducing a scalar multiplication modulo $\g^{(2)}$, say by the scalar $\chi(h)$. This is a $K$-defined subgroup, and $\chi$ is a $K$-defined multiplicative character on $H$. By a straightforward induction, the action of $h\in H$ on $\g^{(i)}/\g^{(i+1)}$ is given by multiplication by $\chi(h)^i$, for all $i\ge 1$. In particular, we have $P_h(h)=0$, where $P_h$ is the polynomial $\prod_{i=1}^\ell(X-\chi(h)^i)$. 
 Since by Rosenlicht's theorem (which applies as $K$ is perfect and infinite), $H^\circ(K)$ is Zariski-dense in $H^\circ(K')$, and since the Zariski-open subset $\Omega=\{h\in H^\circ(K'), \chi(h)^{(2\ell)!}\neq 1\}$ is nonempty (because $H$ is Carnot over $K'$), we can find in $H^\circ(K)\cap\Omega$ some element $g$. Then $\chi(g),\dots,\chi(g)^k$ are pairwise distinct, so $P_g$ has no multiple root and thus $g$ is diagonalizable with eigenvalues $\chi(g),\dots,\chi(g)^\ell$. The corresponding eigenspace decomposition is defined over $K$ and defines a Carnot grading.
\end{proof}

\begin{rem}\label{remdl}
The assertion that if $\g$ is a finite-dimensional Lie algebra over $\mathbf{Q}$ and $\mathbf{R}\otimes_\mathbf{Q}\g$ is Carnot over $\mathbf{R}$, then $\g$ is Carnot, was done by Dekimpe and Lee \cite[Corollary 4.2]{DL}. However, their proof is mistaken. It is based on the assertion \cite[Lemma 4.1]{DL} that in a Carnot Lie algebra, every supplement subspace of $[\g,\g]$ can be chosen to be $\g_1$ for some Carnot grading. However, this assertion is false. For instance, it fails when $\g$ is the 5-dimensional Lie algebra with basis $(X_1,\dots,X_5)$ and nonzero brackets $[X_1,X_2]=X_3$ and $[X_1,X_3]=X_4$ (note that $X_5$ generates a direct factor). Then there is no Carnot grading for which $\g_1$ is the subspace $\mk{v}$ with basis $(X_1,X_2,X_5+X_3)$. The error in \cite[Lemma 4.1]{DL} is the claim, not verified, that $[\mk{v},\mk{v}]$ and $[\mk{v},[\mk{v},\mk{v}]]$ have an intersection reduced to zero.  
\end{rem}

Let us also mention the following fact, which was asserted for Lie algebras when $S$ is finite by Dekimpe and Lee \cite[Proposition 4.3]{DL}, but with a mistaken proof: indeed their proof consists in proving that $[\g,\g]$ admits an $S$-invariant supplement, but this is not enough (for the same reason as in Remark \ref{remdl}).

\begin{prop}\label{recarno}
Let $\g$ be a finite-dimensional Carnot algebra over the $\mathbf{Q}$-field $K$. Let $S$ be a subgroup of automorphisms of $\g$, with a reductive Zariski closure. Then there exists an $S$-invariant Carnot grading.
\end{prop}
\begin{proof}
We use the conventions of the proof of Theorem \ref{main2}.
The group $\underline{\Aut(\g)}$ naturally acts on the space of self-derivations of $\g$, by $(g\cdot\delta)(x)=g\delta(g^{-1}x)$. This action preserves the affine subspace $\underline{D}$ of self-derivations that induce the identity on $\underline{\g/\g^{(2)}}$, which is by assumption non-empty since $\underline{\g}$ is Carnot. Since the Zariski closure of $S$, which we denote by $\underline{S}$, is reductive and $K$ has characteristic zero, the affine subspace $\underline{D}^{\underline{S}}$ of points in $\underline{D}$ fixed by $\underline{S}$ is non-empty; since it is defined over $K$, it admits a $K$-point, and thus $D^S$ is non-empty as well. This defines (by Lemma \ref{carcarn2}) a Carnot-grading on $\g$ that is $S$-invariant.
\end{proof}

\subsection{Maximal grading}\label{cagr}

Let $K$ be a $\mathbf{Q}$-field (i.e., a field of characteristic zero). Let $\g$ be a $K$-algebra that is finite-dimensional as vector space (with no assumption such as associativity). 

\begin{defn}
The $K$-rank of $\g$, denoted $r$, is the the $K$-rank of $\Aut(\g)$.

The gradings on $\g$ in $\mathbf{Z}^r$ induced by maximal $K$-split tori of $\Aut(\g)$  (as in \S\ref{kcg}) are called {\em $K$-maximal gradings} of $\g$.
\end{defn}

\begin{rem}
The systematic use of this notion of maximal grading in the case of Lie algebras can be found, for instance, in \cite{Fav}, where it is called ``system of weights". 

When there is a definition of inner derivation (e.g., in Lie algebras), or inner automorphism (e.g., in associative algebras), a notion of (inner) maximal grading can be defined in terms of inner derivations, or inner automorphisms. In the case of Lie algebras, it is known as the {\em Cartan grading}. We will not use this inner notion here.
\end{rem}

\begin{prop}
Every $K$-maximal grading on $\g$ is an algebra grading in $\mathbf{Z}^r$, and for each such grading, $\mathbf{Z}^r$ is additively generated by the weights.

They are conjugate under $\Aut(\g)^\circ(K)$ and $\GL_d(\mathbf{Z})$: if $(\g_n)_{n\in\mathbf{Z}^r}$ and $(\g'_n)_{n\in\mathbf{Z}^r}$ are maximal gradings, then there exists $u\in\Aut(\g)^\circ(K)$ and $f\in\GL_d(\mathbf{Z})$ such that $\g'_n=u(\g_{f(n)})$ for all $n\in\mathbf{Z}^d$.

The $K$-maximal gradings are the finest gradings of $\g$ in torsion-free abelian groups, in the sense that for every algebra grading $(\g'_\alpha)_{\alpha\in A}$ of $\g$ in a torsion-free abelian group $A$, there exists a $K$-maximal grading $(\g_n)_{n\in\mathbf{Z}^r}$ on $\g$ and a homomorphism $f:\mathbf{Z}^r\to A$ such that $\g'_\alpha=\bigoplus_{n\in f^{-1}(\{\alpha\})}\g_n$ for all $\alpha\in A$.
\end{prop}
\begin{proof}
The first statement is straightforward. The second statement (uniqueness up to conjugacy) is an immediate consequence of the fact that all maximal $K$-split tori in $G=\Aut(\g)$ are conjugate under $\Aut(\g)^0(K)$. 

For the last statement, it is no restriction to assume that $A$ is generated by the weights; in particular $A$ is finitely generated and we can suppose that $A=\mathbf{Z}^s$. Consider the action of $\GL_1^s$ on $\g$, such that $(\lambda_1,\dots,\lambda_s)$ acts on $\g_n$ by multiplication by $\prod\lambda_i^{n_i}$, for $n=(n_1,\dots,n_s)\in\mathbf{Z}^s$. Since $A$ is generated by weights, this action is faithful. The image of $\GL_1^s$ in $\Aut(\g)^\circ$ is a $K$-split torus $T'$, hence is contained in a maximal $K$-split torus $T$; consider the corresponding maximal grading. Then the embedding $T'\subset T$ induces a surjective homomorphism $f:\mathbf{Z}^r=\sfX(T)\to\sfX(T')$ with the required property.
\end{proof}

Some properties of the maximal grading are more sensitive to the context: for instance, if $K$ is an algebraically closed $\mathbf{Q}$-field and $\g$ is a Lie algebra, then $\g_0$ is always nilpotent.

Various properties such as those stated in the introduction can be restated in terms of maximal gradings. 

\begin{prop}
Let $\g$ be a $K$-algebra that is finite-dimensional as vector space, and endow it with a $K$-maximal grading in $\mathbf{Z}^r$ where $r$ is the $K$-rank of $\g$.
First assume that $K$ is algebraically closed. Then
\begin{enumerate}[(1)]
\item\label{kcharnil2} $\g$ is characteristically nilpotent $\Leftrightarrow$ $r=0$ $\Leftrightarrow$ 0 is the only weight of $\g$ for the maximal grading;
\item\label{invder2} $\g$ is flexible if and only if $\g_0=0$;
\item\label{esfl} $\g$ is essentially flexible if and only if the only characteristic ideal of $\g$ contained in $\g_0$ is zero;
\item\label{cnila} the CNI-radical of $\g$ is the largest characteristic ideal of $\g$ contained in $\g_0$.
\end{enumerate}

When $K$ is an arbitrary $\mathbf{Q}$-field, the first two assertions have the following more general form: 
\begin{enumerate}[(1')]
\item\label{kcharnil} $\g$ has no nonzero $K$-diagonalizable derivation $\Leftrightarrow$ $r=0$ $\Leftrightarrow$ 0 is the only weight of $\g$ for the $K$-maximal grading;
\item\label{invder} $\g$ has an invertible $K$-diagonalizable derivation if and only if $\g_0=0$.
\end{enumerate}
\end{prop}
\begin{proof}
(\ref{kcharnil}') is immediate. For (\ref{invder}'), if $\g$ admits an invertible $K$-diagonalizable derivation, then it induces a grading of $\g$ in the additive group of $K$, with $\g_0=\{0\}$. Since the $K$-maximal grading is a refinement of the latter grading, it also satisfies $\g_0=\{0\}$. Conversely if 0 is not a weight for the $K$-maximal grading, pick a homomorphism $\mathbf{Z}^r\to\mathbf{Z}$ mapping no weight to 0; then 0 is not a weight of the resulting grading in $\mathbf{Z}$, which therefore induces an invertible derivation by multiplication by $k$ on $\g_k$. (\ref{kcharnil2}) and (\ref{invder2}) follow as particular cases.

Let us prove (\ref{cnila}), which admits (\ref{esfl}) as a particular case. There exists a homomorphism $f:\mathbf{Z}^r\to\mathbf{Z}$ not vanishing on any nonzero weight. This defines a self-derivation of $\g$, defined to be multiplication by $f(\alpha)$ on $\g_\alpha$. This derivation is diagonalizable with integral eigenvalues and the kernel of this derivation is $\g_0$ (because the ground field is a $\mathbf{Q}$-field). Hence $\cni(\g)\subset\g_0$. Let $\mkh$ be the largest characteristic ideal contained in $\g_0$. It follows that $\cni(\g)\subset\mkh$. To show the reverse inclusion, suppose that $x\in\mkh$ and assume by contradiction that $x\notin\cni(\g)$. Then there exists a semisimple derivation $D$ such that $x\notin\Ker(D)$. This derivation induces a grading of $\g$ in the additive group of the algebraically closed field $K$, for which the set of weights is the spectrum $\mathrm{Spec}(D)$ of $D$. We can find a homomorphism from the subgroup generated by $\mathrm{Spec}(D)$ to $\mathbf{Z}$ that is injective on $\mathrm{Spec}(D)\cup\{0\}$. This defines a grading $(\g'_n)_{n\in\mathbf{Z}}$ of $\g$ in $\mathbf{Z}$ for which $x\notin\g'_0$. Then there exists a maximal grading $(\g''_\alpha)$ refining this grading, and thus $\g''_0\subset\g'_0$ and hence $x\notin\g''_0$. Since all maximal gradings are conjugate by automorphisms, there exists $\eta\in\Aut(\g)$ such that $\eta(\g''_0)=\g_0$. Hence $\eta(x)\notin\g_0$. This contradicts the fact that $\mkh$ is a characteristic ideal contained in $\g_0$.
\end{proof}

Note that these properties are very sensitive on the ground field. On the other hand, we are going to check, using Corollary \ref{finespli}, that several properties related to positivity are invariant under extension of scalars.

It is convenient to use the following language: when a fixed algebra $\g$ is graded in $\mathbf{R}^d$, let us call a linear form on $\mathbf{R}^r$ positive if it maps all weights of $\g$ to non-negative numbers; let us call an element of $\mathbf{R}^r$ {\em positive} if it is mapped to a positive number by some positive linear form, and non-positive otherwise.

\begin{prop}
Let $\g$ be a finite-dimensional $K$-algebra, endowed with a $K$-maximal grading.
\begin{itemize}
\item $\g$ is contractable
 $\Leftrightarrow$ the closed convex cone generated by weights is salient 
$\Leftrightarrow$ there exists a linear form on $\mathbf{R}^r$ sending all weights to positive numbers 
 $\Leftrightarrow$ all weights of $\g$ are positive.
\item $\g$ is semicontractable 
$\Leftrightarrow$ the interior of the convex hull of weights does not contain 0 $\Leftrightarrow$ all weights lie in a closed linear half-space $\Leftrightarrow$ there exists a nonzero positive linear form on $\mathbf{R}^r$ $\Leftrightarrow$ $\g$ admits at least one positive weight
\item $\g$ is Carnot $\Leftrightarrow$ $\g$ is nilpotent and there exists a linear form on $\mathbf{R}^r$ mapping all principal weights to 1 $\Leftrightarrow$ the affine subspace of $\mathbf{R}^r$ generated by principal weights does not contain 0.
\end{itemize}
\end{prop}
\begin{proof}
If $\g$ is contractable, then it admits a grading in $\mathbf{Z}$ for which all weights are positive. This grading has to be obtained by projection from the maximal grading; hence there exists a homomorphism $\mathbf{Z}^r\to\mathbf{Z}$ mapping all weights to positive numbers. This holds if and only if the convex cone generated by weights is salient. Conversely if there exists such a homomorphism, the resulting grading in $\mathbf{Z}$ has only positive weights and hence $\g$ is contractable. The characterization of semicontractable goes along the same lines (keeping in mind that $\mathbf{R}^r$ is spanned by weights).

If $\g$ is Carnot, then it is nilpotent and fixing a Carnot grading, the resulting homomorphism $\mathbf{Z}^r$ maps all principal weights to 1. Conversely, suppose $\g$ is nilpotent and that some linear map $\mathbf{R}^r\to\mathbf{R}$ maps all principal weights to 1, consider the resulting grading of $\g$ in $\mathbf{R}$; for this grading, 1 is the only principal weight. Since $\g$ is nilpotent, it is generated by the $\g_\alpha$ where $\alpha$ ranges over principal weights, and hence is generated by $\g_1$. Thus $\g$ is Carnot. 
\end{proof}

\begin{defn}\label{contradim}
Let $\g$ be a finite-dimensional $K$-algebra, and fix a maximal $K$-split torus in $\Aut(\g)$. The resulting contractive decomposition $\g=\g_{[0]}\ltimes\g_{[+]}$ (Definition \ref{codec}) is called $K$-contractive decomposition of $\g$. The dimensions of $\g_{[0]}$ and $\g_{[+]}$ are called the {\em uncontracted} and {\em contracted dimensions} of $\g$. 

We say that an algebra grading of $\g$ in $\mathbf{N}$ is fine if $\g=\g_0\ltimes\left(\bigoplus_{n>0}\g_n\right)$ is a $K$-contractive decomposition, or equivalently if $\dim(\g_0)$ is equal to the uncontracted dimension of $\g$.
\end{defn}
Here we use the semidirect product notation since $\g_{[0]}$ is a subalgebra and $\g_{[+]}$ is a bilateral ideal.

\begin{prop}\label{codecon}
Let $\g$ be a finite-dimensional $K$-algebra and $\g=\g_{[0]}\ltimes\g_{[+]}$ a $K$-contractive decomposition. Then
\begin{itemize}
\item $\g$ is contractable if and only if $\g=\g_{[+]}$ (i.e., $\g_{[0]}=\{0\}$)
\item $\g$ is semicontractable if and only if $\g_{[+]}$ is nonzero
\item $\cni^+(\g)$ is the largest characteristic bilateral ideal of $\g$ contained in $\g_{[0]}$, and is also the largest bilateral ideal of $\g$ invariant under $\Aut^\circ(\g)$.
\item $\g$ is essentially contractable if and only if $\g_{[0]}$ 
does not contain any nonzero characteristic bilateral ideal, if and only if it does not contain any bilateral ideal invariant under $\Aut(\g)^\circ(K)$.
\end{itemize}
\end{prop}
\begin{proof}
Everything follows from the definitions, except the facts referring to $\Aut(\g)^\circ=\Aut(\g)^\circ(K)$. Let us check the last characterization of $\cni^+(\g)$ (the last characterization of being essentially contractable following immediately). Let $\mk{n}$ be the maximal $\Aut(\g)^\circ$-invariant bilateral ideal contained in $\g_{[0]}$; we have to show that $\mk{n}$ is a characteristic ideal.

Let $\alpha$ be an automorphism of $\g$. Since the contractive decomposition is unique modulo $\Aut(\g)^\circ$, there exists $\beta\in\Aut(\g)^\circ$ such that $\beta\alpha(\g_{[0]})=\g_{[0]}$. Thus $\beta\alpha(\mk{n})$ is an $\Aut(\g)^\circ$-invariant bilateral ideal contained in $\g_0$; hence we have $\beta\alpha(\mk{n})=\mk{n}$; thus composing by $\beta^{-1}$ we get $\alpha(\mk{n})=\mk{n}$.
\end{proof}

\begin{thm}\label{stabex}Let $K\subset L$ be an extension of $\mathbf{Q}$-fields.
Let $\g$ be a finite-dimensional $K$-algebra; write $\g_K=\g$ and $\g_L=L\otimes_K\g$ (viewed as $L$-algebra). Then
\begin{enumerate}
\item\label{concon} $\g_K$ is contractable if and only if $\g_L$ is contractable
\item\label{semiconsemicon} $\g_K$ is semicontractable if and only if $\g_L$ is semicontractable
\item\label{ecokl} $\g_K$ is essentially contractable if and only if $\g_L$ is essentially contractable
\item\label{gcd} if $\g_K=\g_{[0]}\ltimes\g_{[+]}$ is a $K$-contractive decomposition of $\g_K$, then, denoting $\g_{[*]}^L=L\otimes_K\g_{[*]}$ for $*\in\{0,+\}$, we have that $\g_{[0]}^L\ltimes \g_{[+]}^L$ is a contractive decomposition of the $L$-algebra $\g_L=L\otimes_K\g$; in particular the uncontracted and contracted dimension of $\g$ are invariant by field extension of scalars. 
\item\label{cnil} the uncontractable radical $\cni^+(\g_L)$ is equal to $\cni^+(\g)_L=L\otimes_K\cni^+(\g)$ 
\end{enumerate}
\end{thm}
\begin{proof}
Let $\g$ have dimension $d$, and let $G$ be its automorphism group, viewed as an algebraic $K$-subgroup of $\GL_d$.

Then (\ref{gcd}) immediately follows from Corollary \ref{tfi}. In view of Proposition \ref{codecon}, this immediately yields (\ref{concon}) and (\ref{semiconsemicon}).

By Proposition \ref{codecon}, $\cni^+(\g)$ is the largest $\Aut(\g)^\circ(K)$-invariant bilateral ideal of $\g$ contained in $\g_{[0]}$. 
Thus $\cni^+(\g)=\bigcap_{\alpha\in\Aut(\g)^\circ(K)}\alpha(\g_{[0]})$. In particular, for every $\alpha\in\Aut(\g)^\circ(K)$, we have $\cni^+(\g)_L\subset\alpha(\g_{[0]}^L)$. By Zariski density of $\Aut^\circ(\g)(K)$, we deduce that we have $\cni^+(\g)_L\subset\alpha(\g_{[0]}^L)$. Thus $\cni^+(\g)_L\subset\bigcap_{\alpha\in\Aut(\g)^\circ(L)}\alpha(\g_{[0]}^L)$. Let us show the reverse inclusion. We can write $\cni^+(\g)=\bigcap_{i=1}^n\alpha_i(\g_{[0]})$ with $\alpha_i\in\Aut(\g)^\circ(K)$. Hence we obtain $\cni^+(\g)_L=\bigcap_{i=1}^n\alpha_i(\g_{[0]}^L)$. Thus 
\[\cni^+(\g)_L=\bigcap_{i=1}^n\alpha_i(\g_{[0]}^L)\supset \bigcap_{\alpha\in\Aut(\g)^\circ(L)}\alpha(\g_{[0]}^L)\supset \cni^+(\g)_L,\]
so that all inclusions are equalities. Therefore, again using the characterization of Proposition \ref{codecon}, we obtain $\cni^+(\g_L)=\cni^+(\g)_L$, proving (\ref{cnil}), and (\ref{ecokl}) follows.
\end{proof}

\begin{thm}\label{sinv}
Let $\g$ be a finite-dimensional $K$-algebra and $S\subset\Aut(\g)$ a subgroup of automorphisms with a reductive Zariski closure. Then $\g$ admits an $S$-invariant fine grading in $\mathbf{N}$. 
\end{thm}
\begin{proof}
This follows from Corollary \ref{splifi}.
\end{proof}

It yields the following corollary as particular cases:

\begin{cor}
Under the assumptions of Theorem \ref{sinv},
\begin{itemize}
\item if $\g$ is contractable, then it admits an $S$-invariant positive grading;
\item if $\g$ is semicontractable, then it admits an $S$-invariant non-trivial non-negative grading.
\end{itemize}
\end{cor}

This proves conjectures of Dekimpe and Der\'e \cite[Sec.\ 4 and 5]{DD}.

\begin{rem}[On the contractive decomposition]\label{mag01}
If $\g$ is the 2-dimensional non-abelian Lie algebra, then in the contractive decomposition, both $\g_{[+]}$ and $\g_{[0]}$ are 1-dimensional; for all Lie algebras up to dimension 3, $\g_{[+]}$ is actually equal to the nilpotent radical (for any contractive decomposition).

For nilpotent Lie algebras, the smallest examples for which $\g_{[+]}$ is not either $\{0\}$ or $\g$ are the Lie algebras denoted $\g_{7,1,*}$ in Magnin's classification \cite{Mag}, where $*\in\{01(\mathrm{i}),01(\mathrm{ii}),02,03\}$ (see also Example \ref{wdchd}). For the first two, $\g_{[0]}$ has dimension 2; for the last two, $\g_{[0]}$ has dimension 1; in all these four cases, $\cni^+(\g)$ is zero, and $\g_{[+]}$ is a characteristic ideal, i.e., does not depend on the choice of maximal grading. Nevertheless, it is not true in general that $\g_{[+]}$ is a characteristic ideal, see \S\ref{cnincn}.
\end{rem}

\subsection{An example}\label{cnincn}

Here we give a simple example showing that
\begin{itemize}
\item the CNI-radical of a nilpotent finite-dimensional Lie algebra is not always characteristically nilpotent. 
\item the CNI-radical of a direct product of two Lie algebras can be strictly contained in the product of the CNI-radicals of the given factors.
\item if $\g=\g_{[0]}\ltimes\g_{[+]}$ is a contractive decomposition, then $\g_{[+]}$ is not necessarily a characteristic ideal.
\end{itemize}

\begin{prop}
Let $\mka$ be a nonzero abelian Lie algebra and $\g$ a characteristically nilpotent Lie algebra. Then $\cni(\g\times\mka)=\cni^+(\g\times\mka)=[\g,\g]\times\{0\}$.
\end{prop}
\begin{proof}
Let $d$ be a semisimple derivation. Then $d$ maps the derived subalgebra $[\g,\g]\times\{0\}$ into itself, and maps the center $\mk{z}(\g)\times\mka$ into itself. Since $\g$ is characteristically nilpotent, we have $\mk{z}(\g)\subset [\g,\g]$. Thus $d$ maps $[\g,\g]\times\mka$ into itself. Let $\mk{b}$ be a $d$-stable supplement subspace of $[\g,\g]\times\mka$ in $\g\times\mka$. Thus $\mkh=[\g,\g]\oplus\mk{b}$ is a $d$-stable supplement subspace of $\{0\}\times\mka$. Then $\mkh$ is isomorphic to $\g$ (since both are isomorphic to $(\g\times\mka)/\mka$). Thus $\mkh$ is characteristically nilpotent. Hence $d$ is zero on $\mkh$. In particular, $d$ is zero on $[\g,\g]\times\{0\}$. Thus $[\g,\g]\times\{0\}\subset\cni(\g)$. 

Conversely, it is clear that $\cni^+(\g)\subset\g\times\{0\}$. If $x_0$ belongs to $\g\smallsetminus [\g,\g]$, then there exists a homomorphism $f:\g\to\mka$ such that $f(x_0)\neq 0$, and hence $(x,y)\mapsto (x,y+f(x))$ maps $(x_0,0)$ outside $\g\times\{0\}$, so that $(x_0,0)\notin\cni^+(\g)$ since the latter is a characteristic ideal. Therefore $\cni^+(\g)\subset [\g,\g]\times\{0\}$. Since $\cni(\g)\subset\cni^+(\g)$, this completes the proof.
\end{proof}

\begin{cor} There exist nilpotent Lie algebras over any $\mathbf{Q}$-field whose CNI-radical is not characteristically nilpotent. The minimal dimension in which such Lie algebras exist is 8.
\end{cor}
\begin{proof}
If $\g$ is a nonzero characteristically nilpotent Lie algebra of minimal dimension (namely 7), then $[\g,\g]$ is not characteristically nilpotent, and hence if $\mka$ is abelian and nonzero then $\cni(\g\times\mka)\simeq [\g,\g]$ is not characteristically nilpotent. (According to the classification in dimension 7, $[\g,\g]$ is then of dimension 4 or 5: it is 4 for only one example, namely $\g_{7,0,8}$ in Magnin's classification \cite{Mag}, and 5 for all others including the infinite family.) Thus picking $\g$ of dimension 7 and $\mka$ of dimension 1, we get 8-dimensional examples.

Conversely, since according to the classification, in dimension $\le 6$, every Lie algebra $\g$ has a zero CNI-radical, and in dimension 7 the CNI-radical is always $\{0\}$ or $\g$, the smallest possible dimension of a nilpotent algebra whose CNI-radical is not characteristically nilpotent is $\ge 8$.
 \end{proof}

\section{Counterexamples}\label{secount}

Here we indicate a few properties of Lie algebras, also related to tori of automorphisms, but which are not well-behaved with respect to extensions.

\subsection{Anosov nilmanifolds}\label{anoni}

Let $\Gamma$ be a finitely generated torsion-free nilpotent group. Let $G_\mathbf{Q}$ and $G_\mathbf{R}$ be its rational and real Malcev completions, and $\g_\mathbf{Q}$ and $\g_\mathbf{R}$ the corresponding Lie algebras. An Anosov automorphism of $(G_\mathbf{R},\Gamma)$ is an automorphism of $G$ preserving $\Gamma$, such that the corresponding automorphism of $\g_\mathbf{R}$ has no complex eigenvalue of modulus 1. The group $\Gamma$ is called Anosov if there exists an Anosov automorphism of $(G_\mathbf{R},\Gamma)$. It has long been observed that this only depends on $G_\mathbf{Q}$. More precisely, we have:

\begin{prop}\label{anosovcr}The group $\Gamma$ is Anosov if and only if there exists a $\mathbf{Q}$-defined torus in $\Aut(G_\mathbf{Q})$ that is $\mathbf{R}$-split, $\mathbf{Q}$-anisotropic, and which has no nonzero invariant vector in $\g_\mathbf{Q}$.
\end{prop}
\begin{proof}
Let $\phi$ be an Anosov automorphism of $(G_\mathbf{R},\Gamma)$, and let $\xi$ be the corresponding automorphism of $\g_\mathbf{R}$. Let $A$ be the Zariski closure of $\langle\xi\rangle$. Let $T$ be the maximal torus in $A$. We claim that $T$ is $\mathbf{Q}$-anisotropic: indeed if $V\subset\g_\mathbf{Q}$ is a nonzero eigenspace of the maximal $\mathbf{Q}$-split torus in $A$, then $\xi$ preserves a lattice in $V$, and hence acts with determinant one, which implies that the eigenspace is for the trivial character. Let $D$ be the maximal $\mathbf{R}$-split torus in $T$, and let $W$ be the set of vectors in $\g_\mathbf{R}$ fixed by $D$; since $A$ is abelian, it preserves $W$. Then the action of $T_\mathbf{R}$ on $W$ being trivial on $D$, it factors through $(A/D)_\mathbf{R}$, and $A/D$ is unipotent-by-($\mathbf{R}$-anisotropic), and hence is by distal matrices (i.e., with all eigenvalues of modulus 1). Since $\xi$ has no such eigenvalues, this forces $W=\{0\}$. Hence $D$ is as required.

Conversely, let $T$ be an $\mathbf{R}$-isotropic, $\mathbf{Q}$-anisotropic torus with no nonzero invariant vector in $\g_\mathbf{Q}$. Fix an identification of $\g_\mathbf{Q}$ with $\mathbf{Q}^d$, so that $T(\mathbf{Z})$ makes sense. Then $T(\mathbf{Z})$ is a lattice in $T(\mathbf{R})$, and is Zariski-dense in $T$. Then $g\mapsto\det(g^2-1)$ is a nonzero regular map on $T$ and hence does not vanish on $T(\mathbf{Z})$. Let $\xi$ be any element in $T(\mathbf{Z})$ with $\det(\xi^2-1)\neq 0$. Since $\xi$ is diagonalizable over $\mathbf{R}$, this implies that $\xi$ has no complex eigenvalue of modulus 1. Let $\phi$ be the automorphism of $G_\mathbf{Q}$ induced by $\xi$.

To finish the proof, let us show that some power of $\phi$ preserves $\Gamma$. Let $\Gamma'\subset G_\mathbf{Q}$ be a full lattice containing $\Gamma$ as a subgroup (of finite index), where full lattice means that $\log(\Gamma')$ is a Lie subring of $\g_\mathbf{Q}$. In $\g_\mathbf{Q}$, there are finitely many lattices $M$ such that $[M\cap\mathbf{Z}^d:\mathbf{Z}^d][M\cap\mathbf{Z}^d:M]$ is given. In particular, since $\mathbf{Z}^d$ is a fixed point of $\langle\xi\rangle$ for its action on the set of lattices in $\mathbf{Q}^d$, the orbits of the action of $\langle\xi\rangle$ on the set of lattices in $\mathbf{Q}^d$ are finite. Thus some positive power $\xi^m$ stabilizes $\log(\Gamma')$. Thus $\phi^m$ preserves $\log(\Gamma')$ Now we can perform the same argument in the set of lattices of $G_\mathbf{Q}$, so that some positive power $\phi^{mm'}$ preserves $\Gamma$. (The latter argument is a variation of the proof of Lemma 1.1 and Corollary 1.2 in \cite{Dan}.)
\end{proof}

This characterization shows in particular that being Anosov, for finitely generated torsion-free nilpotent groups, is invariant under taking subgroups and overgroups of finite index; this was proved in \cite{Dan}, but the above convenient criterion was not explicitly written there. Thus we can define Anosov for an arbitrary $\mathbf{Q}$-subgroup $H$ of $\GL_d$ by the existence of a subtorus in $H$ as in Proposition \ref{anosovcr}.

\begin{ex}
There exist two lattices in the same simply connected nilpotent Lie group, one being Anosov and the other not. Namely, consider the 3-dimensional Heisenberg group $H_3(\mathbf{R})$ and $G(\mathbf{R})=H_3(\mathbf{R})\times H_3(\mathbf{R})$. Then $G$ admits several $\mathbf{Q}$-forms, one of which is the obvious one, whose $\mathbf{Q}$-points form $H_3(\mathbf{Q})\times H_3(\mathbf{Q})$; this one is not Anosov, because it can be checked that the restriction of its automorphism group to the 2-dimensional center is a $\mathbf{Q}$-split torus, and thus any $\mathbf{Q}$-anisotropic torus acts as the identity on the center. On the other hand, all other $\mathbf{Q}$-forms are Anosov: each such form admits $H_3(\mathbf{Q}[\sqrt{n}])$ as group of $\mathbf{Q}$-points, for some square-free integer $n\ge 2$, and a standard argument \cite{Lau} provides the desired Anosov automorphism: its Lie algebra can be written as $\mk{q}=\mk{h}_3(\mathbf{Q}[\sqrt{n}])$, and can be endowed with a Carnot grading $\mk{q}=\mk{q}_1\oplus\mk{q}_2$ (as a 3-dimensional Lie algebra over $\mathbf{Q}[\sqrt{n}]$). Then we have a torus of $\mathbf{Q}[\sqrt{n}]$-linear automorphisms, acting by multiplication by $t^i$ on $\mk{q}_i$, for $t$ ranging over the group of elements of norm 1 in $\mathbf{Q}[\sqrt{n}]^*$. Now view this as a torus of $\mathbf{Q}$-linear automorphisms of the 6-dimensional Lie $\mathbf{Q}$-algebra $\mk{q}$. This is a $\mathbf{Q}$-anisotropic torus satisfying the conditions of Proposition \ref{anosovcr}.

Also, this example shows that the property that $G$ admits at least one Anosov lattice cannot be read on the complexification of $G$: indeed as just observed, $H_3(\mathbf{R})\times H_3(\mathbf{R})$ admits an Anosov lattice, but it has the same complexification as the real group $H_3(\mathbf{C})$, while the latter admits no Anosov lattice, because its group of automorphisms of determinant 1 acts on the center through an $\mathbf{R}$-anisotropic torus.
\end{ex}

\subsection{Existence of an invertible grading is not invariant under taking extensions}\label{exdere}

\begin{prop}
There exists a finite-dimensional real nilpotent Lie algebra $\mkh$ thatv admits an invertible $\mathbf{Z}$-grading, and admitting a rational form with no such grading.
\end{prop}
\begin{proof}
In \cite[\S 4]{De}, J.\ Der\'e constructs a finite-dimensional real nilpotent Lie algebra as a suitable quotient of the free 6-nilpotent real Lie algebra on 4 generators. Simple computations (see Lemma \ref{checkdere}) show that the derivation Lie algebra of this Lie algebra is solvable and acts on the 4-dimensional abelianization as those diagonal matrices of trace 0. In particular, any maximal torus of the automorphism group is 3-dimensional, $\mathbf{R}$-split, and induces a grading in $\mathbf{Z}^3$ of the Lie algebra.

Der\'e checks \cite[Proposition 4.5]{De} that this Lie algebra admits an Anosov rational form. In particular, it admits an invertible $\mathbf{Z}$-grading over $\mathbf{C}$, and hence over $\mathbf{R}$ since the maximal tori of automorphisms are $\mathbf{R}$-split. A more careful look at his proof shows that the integral points of the torus of automorphisms defining the rational structure contains a discrete abelian group of rank 3. In particular, this torus (for this given $\mathbf{Q}$-form) has to be $\mathbf{Q}$-anisotropic, and hence so are all $\mathbf{Q}$-defined tori of the automorphism group of this $\mathbf{Q}$-form; in particular, this $\mathbf{Q}$-form has no nontrivial grading in $\mathbf{Z}$.
\end{proof}

In the above proof, the (real, but the following works over any field $K$) Lie algebra $\g$ is defined with 4 generators $X_1,\dots,X_4$, killing all commutators of length $\ge 7$, all $[X_i,[X_j,[X_k,X_\ell]]]$ whenever $\{i,j,k,\ell\}=\{1,2,3,4\}$, as well as the four elements obtained as cyclic conjugates of $[X_2,X_4]-[[X_4,[X_3,X_4]],[[X_2,[X_1,X_2]]]$ (note that there are only two such conjugates up to sign).

\begin{lem}\label{checkdere}
The derivations of $\g$ act on $\g/[\g,\g]$ as diagonal matrices in the given basis.
\end{lem}
\begin{proof}
It is enough to prove this for $\g/\g^{(5)}$, which is defined with the generators $X_1,\dots,X_4$, by killing the elements all commutators of length $\ge 5$, killing $[X_1,X_3]$ and $[X_2,X_4]$ and all $[X_i,[X_j,[X_k,X_\ell]]]$ whenever $\{i,j,k,\ell\}=\{1,2,3,4\}$.

Since we only kill monomials, all diagonal matrices on $(X_1,X_2,X_3,X_4)$ extend to derivations of $\g/\g^{(5)}$. The corresponding 4-dimensional diagonal torus defines a grading of $\g/\g^{(5)}$ in $\mathbf{Z}^4$ for which, denoting by $(e_i)$ the canonical basis of $\mathbf{Z}^4$, $X_i$ has degree $e_i$. In turn, this defines a grading on the Lie algebra $\mk{d}$ of derivations of $\g$. Let us check that $\mk{d}_{e_j-e_i}=0$ for all $i\neq j$ (this immediately entails the result). If $f\in\mk{d}_{e_j-e_i}$, then $f(X_i)\in K e_j$, and $f(X_k)=0$ for $k\neq i$. If by contradiction $f\neq 0$, we can assume up to multiplication that $f(X_i)=X_j$. Let $k,\ell$ be the two other elements, and assume that $\ell\neq j+2 \mod 4$. Then 
\[0=f([X_k,[X_j,[X_\ell,X_i]]])=[X_k,[X_j,[X_\ell,X_j]]]\]
We get a contradiction by observing that $[X_k,[X_j,[X_\ell,X_j]]]$ is a nonzero element of $\g/\g^{(5)}$. The latter fact holds because the only elements of degree $2e_j+e_k+e_\ell$ in the ideal generated by the given relators are the scalar multiples of $[X_j,[X_j,[X_k,X_\ell]]]$ if $k=\ell+2 \mod 4$, and the plane generated by $[X_j,X_\ell,[X_j,X_k]]]$ and $[X_\ell,X_j,[X_j,X_k]]]$ if $j=k+2\mod 4$).
\end{proof}

Der\'e directly checks that diagonal derivations of $\g$ have trace 0 in the abelianization, which, unlike the lemma, does not hold for the Carnot Lie algebra $\g/\g^{(5)}$.

\section{Cohopfian nilpotent groups}\label{conig}
\subsection{Grading associated to a rational torus}

Let $A\subset\GL_d=\GL(V)$ a $\mathbf{Q}$-defined abelian subgroup. Hence its identity component $A^\circ$ decomposes as $A_sA_aU$, the almost product of its maximal $\mathbf{Q}$-split and $\mathbf{Q}$-anisotropic subtori, and of its unipotent radical $U$. Then $A$ canonically defines a grading of $V$ in the group of $\mathbf{Q}$-defined multiplicative characters $\sfX(A_s)$, where $V_\chi=\{v\in V:\forall g\in A_s,g(v)=\chi(g)v\}$. The weights of $V$ are by definition those $\chi$ such that $V_\chi\neq\{0\}$.

Beware that the projection $A^\circ\to A_s$ is not well-defined (because of the possible nontrivial finite intersection $A_s\cap A_a$), and hence that characters of $A_s$ do not all define characters of $A$. Still, we consider a projection map $A^\circ(\mathbf{C})\to A_s(\mathbf{C})$, denoted $g\mapsto\hat{g}$, such that $g\hat{g}^{-1}\in (UA_a)(\mathbf{C})$. Then $\hat{g}$ is only determined up to multiplication by an element of the finite group $(A_s\cap A_a)(\mathbf{C})$. 
As a consequence, if $g\in A(\mathbf{C})$, then $\chi(\hat{g})\in\mathbf{C}^*$ is only defined up to multiplication by some root of unity. In particular, its modulus $|\chi(\hat{g})|$ is a well-defined continuous homomorphism, denoted $|\chi|$, from $A^\circ(\mathbf{C})$ to the group $\mathbf{R}_+$ of positive real numbers. The latter extends to a unique continuous homomorphism from $A(\mathbf{C})$ to the group of positive real numbers, again denoted $|\chi|$ (the uniqueness is clear and the existence follows from the injectivity of $\mathbf{R}$ as an abstract $\mathbf{Z}$-module). 

Now let $\xi$ be a fixed element of $A(\mathbf{Q})$. Then we define $\chi\in\sfX(A_s)$ to be non-negative (resp.\ positive, resp.\ distal) if $|\chi|(\xi)\ge 1$ (resp.\ $|\chi|(\xi)>1$, resp.\ $|\chi|(\xi)=1$). We define $V_{\ge 0}$, $V_{>0}$, $V_{\approx 0}$ as the sum of $V_\chi$ where $\chi$ ranges over non-negative, resp.\ positive, resp.\ distal weights. (All these notions are relative to the choice of $\xi$, and do not change if $\xi$ is replaced by a positive power of itself.)

In $\mathbf{Q}^d$, by lattice we mean any subgroup isomorphic to $\mathbf{Z}^d$. Given a group $G$ acting on a set $X$ we say that $g\in G$ stabilizes $Y\subset X$ if $gY\subset Y$ and preserves $Y$ if $gY=Y$.

\begin{lem}\label{v11v2}
Let $\xi$ be an element in $\GL_d(\mathbf{Q})$.
\begin{itemize}
\item
 $\xi$ stabilizes some lattice if and only if there is a bound on denominators in the matrices $\xi^n$ for $n\ge 0$
\item $\xi$ preserves some lattice if and only if there is a bound on denominators in the matrices $\xi^n$ for $n\in\mathbf{Z}$.
\end{itemize}
\end{lem}
\begin{proof}
If $\xi$ stabilizes (resp.\ preserves) a lattice, then up to conjugation we can suppose that this matrix is integral (resp.\ in $\GL_d(\mathbf{Z})$) and hence there is a bound on the denominators as stated. Conversely, if there is a bound on denominators of $\xi^n$ for $n\ge 0$ (resp.\ for $n\in\mathbf{Z}$), the subgroup generated by the $\xi^n\mathbf{Z}^d$ for $n\ge 0$ (resp.\ for $n\in\mathbf{Z}$) is a lattice, and is stabilized (resp.\ preserved) by $\xi$.
\end{proof}

Now consider the above grading, written additively, so that $V_0$ is the set of vectors fixed by $A_s$.

\begin{lem}\label{vao}
If the Zariski closure of $\langle\xi\rangle$ contains $A^\circ$, then $V_{\approx 0}=V_0$.
\end{lem}
\begin{proof}It is no restriction to assume that $\xi\in A^\circ$. Let $\chi$ be a distal weight and let us show that $\chi=0$. Then $A_aU$ acts on $V_\chi$ with determinant 1. For $g\in A^\circ$, write $\check{g}=g\hat{g}^{-1}$. Then $\hat{g}$ acts on $V_{\chi}$ by multiplication by $\chi(\hat{g})$, while $\check{g}$ acts on $V_{\chi}$ with determinant 1. In particular, if $\delta=\dim(V_\chi)$, the determinant of the action of $A^\circ$ on $V_\chi$ is given by $g\mapsto\chi(\hat{g})^{\delta}$, which has to be a rational number whenever $g\in A^\circ(\mathbf{Q})$. Since $\chi$ is distal, we have $1=|\chi|(\xi)=|\chi(\hat{\xi})|$. Thus $\chi(\hat{\xi})^\delta$ is a rational number of modulus 1 and hence is equal to $\pm 1$, and hence $\chi(\hat{\xi})^{2\delta}=1$. Since $\chi$ is also distal with respect to $\xi^n$ for all $n\in\mathbf{Z}$, this shows that $\chi(\hat{\xi^n})^{2\delta}=1$ for all $n\in\mathbf{Z}$. In other words, for all $n\in\mathbf{Z}$, we have $\det(\xi^n|_{V_\chi})^{2}=1$. By Zariski density, we deduce that for all $g\in A^\circ$, we have $\det(g|_{V_\chi})^{2}=1$. On the other hand, if $g\in A_s$, we have $\det(g|_{V_\chi})^2=\chi(g)^{2\delta}$. Hence $\chi(g)^{2\delta}=1$ for all $g\in A_s$. By connectedness of $A_s$, we deduce $\chi(g)=1$ for all $g\in A_s$, which in additive notation means that $\chi=0$.
\end{proof}

\begin{lem}\label{v1v1}
\begin{itemize}
\item if $\xi$ stabilizes a lattice then $V=V_{\ge 0}$
\item if $\xi$ preserves a lattice then $V=V_{\approx 0}$
\end{itemize}

If $\xi$ stabilizes a lattice $\Lambda$, the converse of the second implication holds; more precisely if $V=V_{\approx 0}$ then $\xi$ preserves $\Lambda$. 
\end{lem}
\begin{proof}
Let us show the implications. It is no restriction to assume that $V=V_\chi$ for some $\chi\in\sfX(A_s)$. If $\xi$ preserves a lattice then $\det(\xi)=\pm 1$.
Thus $\det(\hat{\xi})=\pm 1$. Since $\hat{\xi}$ acts by scalar multiplication by $\chi(\hat{\xi})$, the latter is a root of unity and hence $|\chi|(\xi)=1$, which means by definition that $\chi$ is a distal weight. Similarly, if $\xi$ stabilizes a lattice, then $\det(\xi)$ is a nonzero integer, and hence $|\det(\hat{\xi})|=|\det(\xi)|$ is a positive integer and hence is $\ge 1$, which implies that $|\chi|(\xi)\ge 1$, which means by definition that $\chi$ is a non-negative weight.

Now let us show the partial converse. Assume that $V=V_{\approx 0}$ and that $\xi$ stabilizes the lattice $\Lambda$. Again, we can suppose that $V=V_\chi$; then $\chi$ is distal. The assumption implies that $\det(\xi)$ is a nonzero integer $m$. On the other hand, $\det(\xi)=\det(\hat{\xi})$, and $\hat{\xi}$ is the scalar multiplication by $\chi(\hat{\xi})$. Hence $\chi(\hat{\xi})^\delta=m$, where $\delta=\dim(V_\chi)$. Since $\chi$ is distal, $|\chi(\hat{\xi})|=1$. It follows that $m=\pm 1$. Since $\xi$ stabilizes the lattice $\Lambda$ and $\det\xi=\pm 1$, we deduce that $\xi$ preserves $\Lambda$.
\end{proof}

\begin{rem}The converse of both implications of Lemma \ref{v1v1} are false. For instance, let $\xi$ be the companion matrix of $X^2+(1/2)X+1$. Then $\xi$ is irreducible on $\mathbf{Q}^2$, and $\det(\xi)=1$; hence the Zariski closure $A$ of $\langle\xi\rangle$ is a $\mathbf{Q}$-anisotropic torus; thus $V=V_0$ for this choice of $A$. But the spectral radius of $\xi$ and $\xi^{-1}$ in the 2-adic field $\mathbf{Q}_2$ is greater than 1. In particular, $(\xi^n)_{n\ge 0}$ has unbounded denominators (actually, large powers of 2) and thus $\xi$ stabilizes no lattice. 
\end{rem}

\begin{prop}\label{interlati}
Assume that the Zariski closure of $\langle\xi\rangle$ contains $A^\circ$, and that $\xi$ stabilizes a lattice $\Lambda$ in $\mathbf{Q}^d$. Let $A$ be the Zariski closure of $\langle\xi\rangle$. Then $\bigcap_{n\ge 0}\xi^n(\Lambda)$ is equal to $\Lambda\cap V_0$ and is preserved by $\xi$. Moreover, if $\Lambda'$ is another lattice in $\mathbf{Q}^d$, then $\bigcap_{n\ge 0}\xi^n(\Lambda')$ is a lattice in $V_0$. 
\end{prop}
\begin{proof}
Let $\xi'$ be a positive power of $\xi$ contained in $A^\circ$. If the result is proved for $\xi'$ then it immediately follows for $\xi$. Hence we can suppose that $\xi\in A$. That $\Lambda\cap V_0$ is preserved by $\xi$ follows from the partial converse statement in Lemma \ref{v1v1}. Hence it remains to check that $\bigcap\xi^m(\Lambda)\subset V_0$. It is enough to assume that $V=V_\chi$ for $\chi$ not distal and show that $\Lambda'=\bigcap_{n\ge 0}\xi^n(\Lambda)=\{0\}$. Otherwise, let $W$ be the subspace spanned by $\Lambda'$; then $W=W_\chi$ and since $\chi$ is not distal, it is positive, i.e., $|\chi(\hat{\xi})|>1$. Thus $|\det(\xi)|=|\det(\hat{\xi})|=|\chi(\hat{\xi})|^{\dim(W)}>1$ since $\dim(W)>0$; this is a contradiction since $\xi$ preserves the lattice $\Lambda'$ of $W'$. 

For the statement about another lattice, it is clear if $\Lambda'=r\Lambda$ for some nonzero rational number $r$, and the general case follow since any lattice $\Lambda$ is contained in $(1/m)\Lambda$ and contains $m\Lambda$ for some positive integer $m$.
\end{proof}

Now let again $A$ and $\xi$ be as above. We have a grading of $V$ in $\sfX(A_s)$; through the homomorphism $\chi\mapsto\log(|\chi|(\xi))$ from $\sfX(A_s)$ to $\mathbf{R}$, we obtain a new grading of $V$ in $\mathbf{R}$, which we call the absolute grading of $V$ associated to $(A,\xi)$. Actually, since given $\chi$, the value of $|\chi|(\xi)$ only depends on the restriction of $\chi$ to the Zariski closure of $\langle\xi\rangle$, we see that this absolute grading only depends on $\xi$ (and not on the subgroup $A$ containing $\xi$), so we call it the absolute grading of $V$ associated to $\xi$.
 To avoid confusion here, we denote it by $(V_r^\sharp)_{r\in\mathbf{R}}$. Then $V_0^\sharp=V_{\approx 0}$; in particular if $\xi$ generates a Zariski-dense subgroup of $A$ then $V_0^\sharp=V_0$ (Lemma \ref{vao}). It follows from the definition that all weights in the absolute grading are non-negative real numbers if and only if $V=V_{\ge 0}$, which holds if $\xi$ stabilizes a lattice, by Lemma \ref{v1v1}.

If we write $V_r^\sharp=V_r^{\sharp}(\xi)$ to emphasize $\xi$, we can note that $V_{r}^{\sharp}(\xi^m)=V_{mr}^{\sharp}(\xi)$. In particular, $V_0^\sharp$ does not change if $\xi$ is replaced by a nontrivial power. 

\begin{lem}
If $\xi$ is an automorphism of a finite dimensional algebra $\g$ over $\mathbf{Q}$, then the absolute grading defined by $\xi$ is an algebra grading.
\end{lem}
\begin{proof}
If $A$ is the Zariski closure of $\langle\xi\rangle$, then $A\subset\Aut(\g)$, and then the grading defined by $A_s$ is an algebra grading. Since the absolute grading of $\xi$ is a projection of this grading by a group homomorphism, it is also an algebra grading.
\end{proof}

\begin{prop}\label{intersect}
Let $\Gamma$ be a torsion-free finitely generated nilpotent group, $G$ its rational Malcev completion, and $\g$ the Lie algebra of $G$. Let $\phi$ be an injective endomorphism of $\Gamma$, and $\Phi$ the automorphism of $G$ extending $\phi$. Let $\xi$ be the corresponding automorphism of the Lie algebra $\g$, and endow $\g=\bigoplus_{r\in\mathbf{R}}\g_r^\sharp$ with the corresponding absolute grading. Let  $G_{\{1\}}=G_{\{1\}}(\phi)$ the rational subgroup of $G$ corresponding to $\g_{0}^\sharp$. Then $\bigcap_{n\ge 0}\phi^n(\Gamma)$ is equal to $\Gamma\cap G_{\{1\}}$.
\end{prop}
\begin{proof}
We first check that $\bigcap_{n\ge 0}\phi^n(\Gamma)\subset G_{\{1\}}$. The subset $\log(\Gamma)$ is contained in some lattice $L$ of $\g$, and by Proposition \ref{interlati}, we have $\bigcap_{n\ge 0}\xi^n(L)\subset\g_0=\g_{\{0\}}^\sharp$. Hence $\bigcap_{n\ge 0}\phi^n(\Gamma)\subset G_{\{1\}}$. 

On the other hand, since $\phi(\Gamma\cap G_{\{1\}})\subset \Gamma\cap G_{\{1\}}$ and the restriction of $\xi$ to $\g_0$ has determinant $\pm 1$, this inclusion is an equality. Accordingly $\Gamma\cap G_{\{1\}}$ is contained in $\bigcap_{n\ge 0}\phi^n(\Gamma)$.
\end{proof}

\subsection{A lemma for constructing endomorphisms of lattices}
By unipotent $\mathbf{Q}$-group, we mean the group of $\mathbf{Q}$-points of some unipotent linear algebraic $\mathbf{Q}$-group, or equivalently the group obtained from a finite-dimensional nilpotent Lie algebra over $\mathbf{Q}$ when it is endowed with the group law defined by Baker-Campbell-Hausdorff formula. The dimension of a $\mathbf{Q}$-group is defined in the obvious way.
By lattice in a $\mathbf{Q}$-group $G$, we mean any finitely generated subgroup of Hirsch length equal to $\dim(G)$; every unipotent $\mathbf{Q}$-group admits lattices, and they are all commensurate to each other. Also note that if $G_\mathbf{R}$ is the real completion of $G$, then lattices of $G$ are the lattices of $G_\mathbf{R}$ (in the usual sense) that are contained in $G$.

\begin{lem}\label{defendo}
Let $\g$ be a finite-dimensional nilpotent Lie algebra over $\mathbf{Q}$, graded in the non-negative integers $\mathbf{N}$. For $t\in\mathbf{Q}$, let $\delta(t)$ be the automorphism of $\g$ defined as the multiplication by $t^n$ on $\g_n$.

Let $G$ be the corresponding unipotent $\mathbf{Q}$-group and also denote by $\delta(t)$ the corresponding automorphism of $G$.

Let $\Gamma$ be a lattice in $G$. Then there exists an integer $m\ge 2$ such that $\delta(m)$ stabilizes $\Gamma$, in the sense that $\delta(m)(\Gamma)\subset\Gamma$. More precisely, there exists $k_0$ such that every integer $m\in k_0\mathbf{N}+1$ satisfies this condition.
\end{lem}

\begin{proof}
It is convenient to identify the Lie group to the Lie algebra through the exponential map; thus the group law $*$ is given by the Baker-Campbell-Hausdorff (BCH) formula, and the group powers are given by $x^{*m}=mx$ for all $m\in\mathbf{Z}$. Moreover the identification of $\delta(t):\g\to\g$ and $\delta(t):G\to G$ is coherent with this identification.

Fix a basis $(e_i)$ of $\g$, with $e_i$ homogeneous of degree $n_i$; we can multiply these basis elements by large enough integers so as to ensure that all structural constants of the Lie algebra are integers.
Using this basis, identify $\g$ with $\mathbf{Q}^d$. There exist integers $k,k'\ge 1$ such that \[k'\mathbf{Z}^d\subset\Gamma\subset k^{-1}\mathbf{Z}^d.\]

Fix $g\in\g$ and an integer $m\ge 0$. We need to describe $g*(\delta(m+1)(g))^{*-1}=g*(-\delta(m+1)(g))$; it is given by the BCH-formula as a certain sum $g-\delta(m+1)(g)+h$, where $h$ is a $\mathbf{Z}[1/s]$-linear combination of iterated brackets $[c_1,\dots,c_\ell,[g,\delta(m+1)(g))]\dots]$, where $0\le\ell\le d-2$ (actually $\ell$ is at most the nilpotency length minus 2), and $s$ is a common denominator for terms in the BCH-formula of degree at most the nilpotency length of $\g$ (we can choose $s=d!$ although it is far from optimal), and the $c_i$ are in $\{g,\delta(m+1)(g)\}$.

We write $g=\sum_{n\ge 0}g_n$, according to the decomposition $\g=\g_0\oplus\left(\bigoplus_{n\ge 1}\g_n\right)$. Then

\[\delta(m+1)(g)=\bigoplus_{n\ge 0} (1+m)^ng_n=g+ \bigoplus_{n\ge 0} ((1+m)^n-1)g_n.\]
Since $m$ divides $((1+m)^n-1)$ for all $n\ge 0$, we can write $\delta(m+1)(g)=g+mg'$. Here $g'=\bigoplus_{n\ge 1}\frac{(1+m)^n-1}{m}g_n$. Hence $g-\delta(m+1)(g)=-mg'$, and $[g,\delta(m+1)(g)]=m[g,g']$.

Now assume that $g\in\Gamma$, so that $g$ belongs to $k^{-1}\mathbf{Z}^d$. It also follows that all $g_n$ belong to $k^{-1}\mathbf{Z}^d$, and hence $\delta(m+1)(g)$ and $g'$ belong to $k^{-1}\mathbf{Z}^d$ as well. Hence $[c_1,\dots,c_\ell,[g,\delta(m+1)(g))]\dots]$ belongs to $mk^{-d}\mathbf{Z}^d$. Therefore, if we assume that $sk^dk'$ divides $m$, we deduce that $g*(\delta(m+1)(g))^{*-1}$ belongs to $k'\mathbf{Z}^d$ and hence belongs to $\Gamma$. Since $g\in\Gamma$ as well, it follows that $\delta(m+1)(g)\in\Gamma$. Hence we get the conclusion with $k_0=sk^dk'$.
\end{proof}

\subsection{The main theorem about cohopfian properties}
\begin{thm}\label{maincoh}
Let $\Gamma$ be a torsion-free finitely generated nilpotent group, $G$ its rational Malcev completion, and $\g$ the Lie algebra of $G$. Let $k$ be the uncontracted dimension of $\g$ (Definition \ref{contradim}). Then the smallest possible Hirsch length for $\bigcap_{n\ge 0}\phi^n(\Gamma)$, when $\phi$ ranges over injective endomorphisms of $\Gamma$, is $k$. Actually, if $\g=\g_{[0]}\ltimes\g_{[+]}$ is a contractive decomposition of $\g$ and $G_{[0]}$ is the subgroup of $G$ corresponding to $\g_{[0]}$, then there exists an injective endomorphism $\phi$ of $G$ such that $\bigcap\phi^n(\Gamma)=G_{[0]}\cap\Gamma$.
\end{thm}
\begin{proof}
Let $\phi$ be an injective endomorphism of $\Gamma$. Then the Hirsch length $h(\Gamma,\phi)$ of $\bigcap_{n\ge 0}\phi^n(\Gamma)$ is, in the notation of Proposition \ref{intersect} and using this proposition, the dimension of $\g_0^\sharp$. Let $M$ be the subgroup of $\mathbf{R}$ generated by weights of $\g$ in the absolute grading defined by $\xi$. This is a finitely generated torsion-free abelian group; hence it has a basis. Then by a small perturbation of those basis elements, we find a homomorphism $M\to\mathbf{R}$ mapping all positive weights of $\g$ to positive rational numbers, and then after multiplication we find a homomorphism $M\to\mathbf{R}$ mapping all weights to positive integers. Thus we have a Lie algebra grading of $\g$ in $\mathbf{N}$, denoted $(\g_n)_{n\in\mathbf{Z}}$, for which $\g_0=\g_0^\sharp$. This implies that there exists a contractive decomposition $\g=\g_{[0]}\ltimes\g_{[+]}$ of $\g$ such that $\g_0\supset\g_{[0]}$. Hence $h(\Gamma,\phi)=\dim(\g_0^\sharp)\ge k$.

To show that $k$ is achieved, let $\g=\g_{[0]}\ltimes\g_{[+]}$ be a contractive decomposition of $\g$: then there exists a Lie algebra grading of $\g$ in $\mathbf{N}$ such that $\g_0=\g_{[0]}$. Define $\delta(t)$ as in Lemma \ref{defendo}, namely to be the multiplication by $t^i$ on $\g_i$. Then this lemma asserts that there exists an integer $m\ge 2$ such that $\delta(m)$ maps $\log(\Gamma)$ into itself. Thus the corresponding automorphism $\Delta_m=\exp\circ\delta(m)\circ\log$ of $G$ maps $\Gamma$ into itself. Then $\bigcap_{n\ge 0}\Delta_m^n(\Gamma)=\Gamma\cap G_{[0]}$: although this is a particular case of Proposition \ref{intersect}, this can be seen directly since $\delta(m)$ is diagonalizable over $\mathbf{Q}$. Hence $h(\Gamma,\Delta_m)=k$.
\end{proof}

\begin{cor}\label{corno}
Let $\Gamma$ be a finitely generated torsion-free nilpotent group and $\g$ its rational Lie algebra. Then
\begin{itemize}
\item $\Gamma$ is non-cohopfian if and only if $\g$ is semi-contractable;
\item $\Gamma$ is dis-cohopfian if and only if $\g$ is contractable.
\end{itemize}
\end{cor}
\begin{proof}
If $\phi$ is an injective endomorphism of $\Gamma$, define $h(\Gamma,\phi)$ as the Hirsch length of $\bigcap\phi^n(\Gamma)$, and $k$ as the uncontracted dimension of $\g$, so the theorem asserts that $k=\min_\phi(\Gamma,\phi)$. 

Then $\g$ is semi-contractable if and only if $k<\dim(\g)$, and is contractable if and only if $k=0$. On the other hand, we have $h(\Gamma,\phi)<\dim(\g)$ if and only if $\phi$ is non-surjective (observing that $\bigcap\xi^n(\Gamma)$ cannot be a proper subgroup of finite index), and $h(\Gamma,\phi)=0$ if and only if $\bigcap\phi^n(\Gamma)=\{1\}$. Thus the theorem implies the corollary.
\end{proof}

\begin{proof}[Proof of Theorem \ref{crcc}]
(\ref{nch2})$\Rightarrow$(\ref{nch3}) is clear, since $G$ admits a lattice.

(\ref{nch3})$\Rightarrow$(\ref{nch1}). Let $\Gamma$ be a non-cohopfian lattice. Let $G_\mathbf{Q}$ be the group of $\mathbf{Q}$-points for the $\mathbf{Q}$-form defined by $\Gamma$, and $\g_\mathbf{Q}$ its Lie algebra. Then by Corollary \ref{corno}, $\g_\mathbf{Q}$ is semi-contractable. Thus $\g\simeq\mathbf{R}\otimes_\mathbf{Q}\g_\mathbf{Q}$ is semi-contractable.

(\ref{nch1})$\Rightarrow$(\ref{nch2}) Let $\Gamma$ be a lattice. Let $G_\mathbf{Q}$ be the group of $\mathbf{Q}$-points for the $\mathbf{Q}$-form defined by $\Gamma$, and $\g_\mathbf{Q}$ its Lie algebra. Then since $\g\simeq\mathbf{R}\otimes_\mathbf{Q}\g_\mathbf{Q}$ and $\g$ is semi-contractable, by Theorem \ref{stabex} it follows that $\g_\mathbf{Q}$ is semi-contractable. Thus by Corollary \ref{corno}, $\Gamma$ is not cohopfian.
\end{proof}

\begin{proof}[Proof of Theorem \ref{crcd}]
The proof follows mutatis mutandis the same steps as the previous ones, replacing ``non-cohopfian" with ``dis-cohopfian" and ``semicontractable" with ``contractable".
\end{proof}

\subsection{Generalities about being weakly dis-cohopfian}\label{gwdh}

For short, call $\Gamma$-chain any descending sequence (not necessarily with strict inclusions) of subgroups $(\Gamma_n)$ of $\Gamma$, all isomorphic to $\Gamma$, with $\Gamma_0=\Gamma$. Note that for each such chain, there exists a sequence $(\phi_n)_{n\ge 1}$ of injective endomorphisms of $\Gamma$ such that $\Gamma_n=\phi_1\dots\phi_n(\Gamma)$ for all $n$.

\begin{lem}\label{sandwich}
Let $\Gamma$ be a countable group and $W$ a subgroup. Assume that
\begin{enumerate}
\item\label{fory} for every injective endomorphism $\phi$ of $\Gamma$, we have $\phi(W)=W$, and 
\item\label{exin} for every $\gamma\notin W$, there exists an injective endomorphism $\phi$ of $\Gamma$ such that $\gamma\notin\phi(\Gamma)$.
\end{enumerate}
Then the following two assertions hold
\begin{enumerate}[(a)]
\item\label{surrr} for every $\Gamma$-chain $(\Gamma_n)$ we have $W\subset\bigcap\Gamma_n$;
\item\label{exgc} there exists a $\Gamma$-chain $(\Gamma_n)$ such that $\bigcap\Gamma_n=W$.
\end{enumerate}
\end{lem}
\begin{proof}
(\ref{surrr}) is an immediate consequence of (\ref{fory}). Let us prove (\ref{exgc}). If $W=\Gamma$, there is nothing to do; otherwise, let $(g_n)_{n\ge 1}$ be a (possibly non-injective) enumeration of $\Gamma\smallsetminus W$, and define by induction a $\Gamma$-chain $(\Gamma_n)$ such that $g_n\notin\Gamma_i$ for all $n\ge 1$. By definition $\Gamma_0=\Gamma$. Assume $n\ge 1$ and $\Gamma_i$ is defined for $i<n$. If $g_n\notin\Gamma_{n-1}$, we choose $\Gamma_n=\Gamma_{n-1}$. Otherwise we apply (\ref{exin}) to $\Gamma_{n-1}$ and $\gamma=g_n$ to define $\Gamma_n$ as the image of an injective endomorphism of $\Gamma_{n-1}$ whose image does not contain $g_n$. Since all $\Gamma_n$ contain $W$ by (\ref{fory}) and no $g_n$ belongs to $\bigcap_i\Gamma_i$, we obtain $\bigcap_i\Gamma_i=W$. Accordingly, $(\Gamma_n)$ is a $\Gamma$-chain with the required property. 
\end{proof}

\begin{cor}\label{carwdisch}
A countable group $\Gamma$ is weakly dis-cohopfian if and only if for every $g\in\Gamma\smallsetminus\{1\}$ there exists a subgroup of $\Gamma$ isomorphic to $\Gamma$ and not containing $g$.
\end{cor}
\begin{proof}
This is the particular case of Lemma \ref{sandwich}, when $W=\{1\}$, noting that (\ref{fory}) then holds automatically.
\end{proof}

\begin{rem}
A variant of Lemma \ref{sandwich} holds, namely when in both the assumptions and conclusions, ``injective endomorphism" is replaced with ``injective endomorphism with image of finite index" and ``$\Gamma$-chain" is replaced with ``$\Gamma$-chain $(\Gamma_n)$ such that all $\Gamma_n$ have finite index, the proof being the same.
\end{rem}

\subsection{Weakly dis-cohopfian nilpotent groups}
(Refer to \S\ref{gwdh} for the meaning of $\Gamma$-chain.)

\begin{thm}
Let $\Gamma$ be a finitely generated torsion-free nilpotent group, $G$ its rational Malcev completion, and $\g$ the Lie algebra of $G$. Let $\cni^+(G)$ be the subgroup corresponding to $\cni^+(\g)$. Then for every $\Gamma$-chain $(\Gamma_n)$, we have $\bigcap_n\Gamma_n\supset\cni^+(G)\cap\Gamma$, with equality for some choice of $(\Gamma_n)$.
\end{thm}
\begin{proof}
Let $\phi$ be an injective endomorphism of $\Gamma$, $\hat{\phi}$ its unique extension to an automorphism of $G$, and $\xi=\log\circ\hat{\phi}\circ\exp$ the corresponding automorphism of $\g$. Since $\xi$ preserves $\cni^+(\g)$, $\hat{\phi}$ preserves $\cni^+(G)$. Hence $\phi$ stabilizes $\cni^+(G)\cap\Gamma$. Since the determinant of the restriction of $\phi$ to $\cni^+(\g)$ is $\pm 1$, it follows that $\phi$ preserves $\cni^+(G)\cap\Gamma$. Accordingly, $\cni^+(G)\cap\Gamma\subset\bigcap_n\Gamma_n$.

Let us now show the existence statement. Let $\g=\g_{[0]}\ltimes\g_{[+]}$ be a contractive decomposition of $\g$, and $G_{[0]}$ the subgroup corresponding to $\g_{[0]}$. Then by Theorem \ref{maincoh}, there exists an injective endomorphism $\phi$ of $\Gamma$ such that $\bigcap_{n\ge 0}\phi^n(\Gamma)$ is equal to $G_{[0]}\cap\Gamma$.

Define $W=\cni^+(G)\cap\Gamma$. By the previous verification, $\phi'(W)=W$ for every injective endomorphism $\phi'$ of $\Gamma$, which is the assumption (\ref{fory}) of Lemma \ref{sandwich}. Let us check (\ref{exin}) of the same lemma. Fix $x\in\Gamma\smallsetminus W$. Thus $x\notin\cni^+(G)$, that is, $\log(x)\notin\cni^+(\g)$. Thus there exists a contractive decomposition of $\g$ such that $\log(x)\notin\g_{[0]}$. Given a non-negative Lie algebra grading of $\g$ such that $\g_0=\g_{[0]}$, Lemma \ref{defendo} provides an injective endomorphism $\xi$ of $\Gamma$ such that $\bigcap\xi^n(\Gamma)\subset\exp(\g_0)$. In particular $x$ does not belong to the image of $\xi^n$ for some $n$. This means that (\ref{exin}) holds, and hence Lemma \ref{sandwich} implies that for some $\Gamma$-chain $(\Gamma_n)$ we have $\bigcap\Gamma_n=W$.
\end{proof}

\begin{cor}\label{corwdi}
Let $\Gamma$ be a finitely generated torsion-free nilpotent group, and $\g$ its rational Lie algebra. Then $\Gamma$ is weakly dis-cohopfian if and only if $\cni^+(\g)=\{0\}$.
\end{cor}

\begin{proof}[Proof of Theorem \ref{crcw}]
Again, the proof follows mutatis mutandis from that of Theorem \ref{crcc}, replacing the use of Corollary \ref{corno} by Corollary \ref{corwdi}.
\end{proof}

\begin{ex}\label{wdchd}
Let $\g$ be one of the four complex 7-dimensional Lie algebras referred to in Remark \ref{mag01}. For instance, the Lie algebra $\g_{7,1.02}$ can be defined from a basis $(Z_1,A_2,A_3,A_4,B_5,B_6,C_7)$ by the nonzero brackets
$$[Z_1,A_2]=A_3, \quad [Z_1,A_3]=A_4, \quad [A_2,A_3]=B_5, \quad [Z_1,B_5]=B_6,$$ $$[A_2,A_4]=B_6,\quad [A_2,B_5]=C_7,\quad [A_2,B_6]=C_7,\quad [A_3,B_5]=-C_7.$$

(The grading on $\mathbf{N}$ for which $Z_1$ has degree 0, $A_i$ has degree 1, $B_i$ has degree 2, and $C_7$ has degree 3, is a maximal grading.)

These Lie algebras being defined using rational coefficients, we can find a lattice $\Gamma$ in the real group. Then $\cni^+(\g)=\{0\}$ but $\g$ is not contractable. Thus $\Gamma$ is weakly dis-cohopfian but not dis-cohopfian.
\end{ex}

\section{Systolic growth and geometry of lattices in nilpotent groups}\label{gelani}

\subsection{Generalities on the systolic growth}

If $f,g$ are non-negative functions defined on the integers or reals, we say that $f\preceq g$, or that $f$ is asymptotically bounded above by $g$, if there exists a positive constant $C$ such that $f(x)\le Cg(Cx)+C$ for all $x$ large enough. If $f\preceq g\preceq f$, we say that $f,g$ are asymptotically 
equivalent and write $f\simeq g$.

Also, we write $f\sim g$ if $f=O(g)$ and $g=O(f)$ (for instance, $2^n\simeq 3^n$ but $2^n\nsim 3^n$).

We first generalize the notion of systolic growth from the introduction to compactly generated locally compact groups.

\begin{defn}
Let $G$ be a locally compact group. We say that $G$ is {\em residually systolic} if for every compact neighborhood $S$ of $1$, there exists a lattice $\Gamma$ in $G$ such that $\Gamma\cap S=\{1\}$.
\end{defn}

When $G$ is discrete, residually systolic just means residually finite. In general, residually systolic implies the existence of lattices, and in particular implies unimodular. If $G$ admits a residually finite lattice, then it is residually systolic.

Let now $G$ be a locally compact, compactly generated group, endowed with the word length $|\cdot|$ with respect to some compact generating subset, or any equivalent length. We fix a left Haar measure on $G$, so that every discrete subgroup has a well-defined covolume (possibly infinite).

\begin{defn}
If $\Gamma\subset G$, we define its {\em systole} as 
\[\sys(\Gamma)=\inf\{|\gamma|:\gamma\in\Gamma\smallsetminus\{1\}\},\]
with $\inf\emptyset=+\infty$. The {\em systolic growth} of $G$ is the function $\sigma=\sigma_{G,|\cdot|}$ mapping $r$ to the infimum of covolumes of discrete subgroups of $G$ of systole $\ge r$ ($+\infty$ if there is no such subgroup).  
\end{defn}

Note that $\sigma(r)<\infty$ for all $r$ if and only if $G$ is residually systolic, i.e., $G$ admits lattices of arbitrary large systole. The growth is a lower bound for the systolic growth: if $v(r)$ is the volume of the open $r$-ball, then $\sigma(2r)\ge v(r)$. 

Note that an alternative definition would restrict to cocompact lattices; if we focus on nilpotent groups or discrete groups, of course this makes no difference. 

\begin{defn}
If $\Gamma\subset G$, define its {\em normal systole} as
\[\sys^\lhd(\Gamma)=\inf\{|g\gamma g^{-1}|:\gamma\in\Gamma\smallsetminus\{1\},\;g\in G\}.\]
The {\em normal systolic growth} of $G$ is the function $\sigma^\lhd=\sigma^\lhd_{G,|\cdot|}$ mapping $r$ to the infimum of covolumes of discrete subgroups of $G$ of systole $\ge r$ ($+\infty$ if there is no such subgroup).
\end{defn}

Obviously, we have $\sys^\lhd\le\sys$, and $\sigma^\lhd\ge\sigma$. For standard reasons, the asymptotic behaviors of $\sigma$ and $\sigma^\lhd$ do not depend on the choice of word length and Haar measure.

\begin{lem}\label{nsgcoi}
The $\simeq$-asymptotic growth of the normal systolic growth is a commensurability invariant for finitely generated groups.
\end{lem}
\begin{proof}
Suppose $\Gamma'\subset\Gamma$ has finite index. We endow $\Gamma$ with a word length $|\cdot|$, and endow $\Gamma'$ with the restriction of this word length, which is equivalent to the word length on $\Gamma'$, and define the systoles with these lengths.

If $\Lambda$ is a finite index subgroup in $\Gamma$, then $[\Gamma':\Lambda\cap\Gamma']\le [\Gamma:\Lambda]$ and $\sys^\lhd_{\Gamma'}(\Lambda\cap\Gamma')\ge\sys^\lhd_{\Gamma}(\Lambda)$. This implies that $\sigma^\lhd_{\Gamma'}\le\sigma^\lhd_{\Gamma}$.

Let us now prove an inequality in the other direction.
Write $\Gamma=F\Gamma'$ with $F$ finite. If $g\in\Gamma$, define $[g]_\Gamma=\inf_{h\in\Gamma}|hgh^{-1}|$, and $[g]_{\Gamma'}=\inf_{h\in\Gamma'}|hgh^{-1}|$. Then
\[[g]_{\Gamma'}-k\le [g]_\Gamma\le [g]_{\Gamma'},\qquad k=2\sup_{h\in F}|h|\]
It follows that for every finite index subgroup $\Lambda$ of $\Gamma'$, we have
\[\sys^\lhd_{\Gamma'}(\Lambda)-k\le \sys^\lhd_{\Gamma}(\Lambda)\le \sys^\lhd_{\Gamma'}(\Lambda),\]
and this implies $\sigma^\lhd_\Gamma(n-k)\le [\Gamma:\Gamma']\sigma^\lhd_{\Gamma'}(n)$ for all $n\ge k$, thus $\sigma^\lhd_\Gamma\preceq \sigma^\lhd_{\Gamma'}$. Finally we deduce $\sigma^\lhd_\Gamma\simeq \sigma^\lhd_{\Gamma'}$.
\end{proof}

\subsection{The main theorem on systolic growth and first part of the proof}

\begin{thm}\label{corlat2}
Let $G$ be a connected unipotent $\mathbf{Q}$-group. Let
\[\delta=\sum_{i\ge 1}i\dim(G^{(i)}/G^{(i+1)})\] be the growth rate of the simply connected nilpotent Lie group $G_\mathbf{R}$. Let $\g_\mathbf{Q}$ and $\g_\mathbf{R}=\mathbf{R}\ot_\mathbf{Q}\g_\mathbf{Q}$ be the corresponding Lie algebras.  
Let $\sigma$ and $\sigma^\lhd$ be the systolic growth and normal systolic growth of $G_\mathbf{R}$. Equivalences:
\begin{enumerate}[(i)]
\item\label{xcarr} the Lie algebra $\g_\mathbf{R}$ is Carnot over $\mathbf{R}$;
\end{enumerate}
\begin{enumerate}[(i$^+$)]
\item\label{xcarrq} the Lie algebra $\g_\mathbf{Q}$ is Carnot over $\mathbf{Q}$
\end{enumerate}
\begin{enumerate}[(i$^-$)]
\addtocounter{enumi}{1}
\item\label{eliem} either $G=\{1\}$, or every lattice in $G_\mathbf{Q}$ (i.e., lattice of $G_\mathbf{R}$ contained in $G_\mathbf{Q}$) admits an injective endomorphism $\phi$ such that, for some integer $m\ge 2$, we have $\sys(\mathrm{Im}(\phi^n))\sim\sys^\lhd(\mathrm{Im}(\phi^n))\sim m^{n}$ and $\mathrm{covol}(\mathrm{Im}(\phi^n))=\lambda m^{\delta n}$ for some $\lambda>0$ and all $n\in\mathbf{N}$;
\end{enumerate}
\begin{enumerate}[(i)]
\addtocounter{enumi}{1}
\item\label{xelasy} every lattice in $G_\mathbf{Q}$ has systolic growth $\simeq n^\delta$;
\item\label{xslasy} some lattice in $G_\mathbf{Q}$ has systolic growth $\simeq n^\delta$;
\end{enumerate}
\begin{enumerate}[(i$^-$)]
\addtocounter{enumi}{3}
\item\label{sird} $\sigma(r)\simeq r^\delta$; 
\end{enumerate}
\begin{enumerate}[(i)]
\addtocounter{enumi}{3}
\item\label{xecovol} $\liminf\sigma(r)/r^\delta<\infty$; 
\item\label{elnsg} every lattice in $G_\mathbf{Q}$ has normal systolic growth $\simeq n^\delta$;
\item\label{gnsg} $G_\mathbf{R}$ has normal systolic growth $\simeq r^\delta$.
\end{enumerate}
\end{thm}

Note that Theorem \ref{corlat} holds as a consequence. We are going to prove the equivalences of all properties except (\ref{gnsg}) by proving the cycle of implications

(\ref{xcarr})$\stackrel{\bullet}\Rightarrow$(\ref{xcarrq}$^+$)$\stackrel{\bullet}\Rightarrow$(\ref{eliem}$^-$)$\stackrel{\circ}\Rightarrow$(\ref{elnsg})$\stackrel{\circ}\Rightarrow$(\ref{xelasy})$\stackrel{\circ}\Rightarrow$(\ref{xslasy})$\stackrel{\circ}\Rightarrow$(\ref{sird}$^-$)$\stackrel{\circ}\Rightarrow$(\ref{xecovol})$\stackrel{\bullet}\Rightarrow$(\ref{xcarr}),

\noindent which along with the implications (\ref{eliem})$\stackrel{\circ}\Rightarrow$(\ref{gnsg})$\stackrel{\circ}\Rightarrow$(\ref{sird}$^-$) proves all the equivalences.

Here $\stackrel{\bullet}\Rightarrow$ means an implication requiring a proof, that is part of the work of the paper, while $\stackrel{\circ}\Rightarrow$ means an implication which are trivial or only require a few comments, which follow here:

\begin{itemize}
\item (\ref{eliem}$^-$)$\Rightarrow$(\ref{elnsg}) and (\ref{eliem}$^-$)$\Rightarrow$(\ref{gnsg}): formally speaking, (\ref{eliem}$^-$) implies that $n^\delta$ is an asymptotic lower bound for the normal systolic growth of both $G$ and its lattices. Since conversely the growth of both $G$ and its lattices is $\simeq n^\delta$ and is an asymptotic lower bound for the normal systolic growth, this yields the conclusion.
\item (\ref{elnsg})$\Rightarrow$(\ref{xelasy}) and (\ref{gnsg})$\Rightarrow$(\ref{sird}$^-$): just use that the systolic growth is asymptotically trapped between the growth and the systolic growth; similarly, for (\ref{xslasy})$\Rightarrow$(\ref{sird}$^-$), just use that the systolic growth of $G_\mathbf{R}$ is asymptotically trapped between the growth of $G_\mathbf{R}$ and the systolic growth of any of its lattices.
\item (\ref{xelasy})$\Rightarrow$(\ref{xslasy}): this follows from the existence of a lattice, which is known to be equivalent to the assumption that $\g$ is definable over $\mathbf{Q}$.
\end{itemize}

It remains to consider the implications $\stackrel{\bullet}\Rightarrow$. 
Since (\ref{xcarr})$\Rightarrow$(\ref{xcarrq}$^+$) is part of Theorem \ref{main2}, the only two remaining are
(\ref{xcarrq}$^+$)$\Rightarrow$(\ref{eliem}$^-$) and (\ref{xecovol})$\Rightarrow$(\ref{xcarr}).

\begin{proof}[Proof of Theorem \ref{corlat2}, (\ref{xcarrq}$^+$)$\Rightarrow$(\ref{eliem}$^-$)]

Write the Carnot grading as $\g_\mathbf{Q}=\bigoplus (\g_i)_\mathbf{Q}$.
Write $\g=\g_\mathbf{R}$ and $\g_i=(\g_i)_\mathbf{R}$. Fix norms on the $\g_i$, and define, for $x=\sum x_i\in\g=\bigoplus\g_i$, the ``length" $\ell(x)=\sup\|x_i\|^{1/i}$. By Guivarch's estimates \cite{Guiv}, the word length on $G$ and on its lattices are equivalent to $\ell\circ\log$. Therefore in the sequel of the proof, we use $|\cdot|$ to denote $\ell\circ\log$, although it is not necessarily a length (it may fail to be sub-additive, but this does not matter), and use it in the definition of systole. If $x\in\g\smallsetminus\{0\}$, write $i(x)=\min\{i:x_i\neq 0\}$.

We can assume that $\g\neq\{0\}$. Let $\Gamma$ be a lattice in $G_\mathbf{Q}$, and define $\Gamma'=\log(\Gamma)$. 
Define $\delta(m)$ as the automorphism of $\g$ (and $\g_\mathbf{Q}$) defined to be the multiplication by $m^i$ on $\g_i$. 
By Lemma \ref{defendo}, there exists an integer $m\ge 2$ such that $\delta(m)$ maps $\Gamma'$ into itself. Let $\phi=\exp\circ\delta(m)\circ\log$ be the corresponding automorphism of $G$.
Then $\phi(\Gamma)$ has index $m^\delta$ in $\Gamma$, and it follows that $\mathrm{covol}(\phi^n(\Gamma))=m^{n\delta}\mathrm{covol}(\Gamma)$ for all $n\in\mathbf{Z}$. 

It remains to estimate the systole and normal systole of $\phi^n(\Gamma)$. We have $\ell(\delta(m)h)=m\ell(h)$ for all $h\in\g$, and thus $\sys(\phi^n(\Gamma))=m^n\sys(\Gamma)$. This also yields $\sys^\lhd(\phi^n(\Gamma))\le m^n\sys(\Gamma)$. To obtain a lower bound for $\sys^\lhd(\phi^n(\Gamma))$, consider any $g\in\Gamma'\smallsetminus\{0\}$ and write $i=i(g)$, and define $\pi_i$ as the projection on $\g_i$. Then for every $h\in G_\mathbf{R}$, we have $\pi_i(hgh^{-1})=\pi_i(g)$ (we identify the group and the Lie algebra in order to avoid cumbersome notation), and thus $\ell(hgh^{-1})\ge \|g_i\|^{1/i}$. Thus for any $n\ge 0$, any $g\in\Gamma'\smallsetminus\{0\}$, we have $\ell(h\phi^n(g)h^{-1})\ge m^n\|g_i\|^{1/i}$. If we define $k=\inf_{g\in\Gamma'\smallsetminus\{0\}}\|g_{i(g)}\|^{1/i(g)}$, then $k>0$ (the infimum is attained, by an easy argument using that the projection of $\Gamma$ on $G^{(i)}/G^{(i+1)}$ is discrete for all $i$), and we thus have  $\ell(h\phi^n(g)h^{-1})\ge km^n$ for all $h\in G_\mathbf{R}$ and $g\in\Gamma'\smallsetminus\{0\}$. Thus $\sys^\lhd(\phi^n(\Gamma))\ge km^n$ for all $n\in\mathbf{N}$.
\end{proof}

Before proceeding to the last and most difficult implication (\ref{xecovol})$\Rightarrow$(\ref{xcarr}) in \S\ref{gpart}, we provide a few auxiliary results.

\begin{cor}
Let $\Gamma$ be a finitely generated, virtually nilpotent group, of growth $\simeq n^\delta$. Suppose that the rational Lie algebra of some/any torsion-free nilpotent finite index subgroup is Carnot. Then the systolic growth and normal systolic growth of $\Gamma$ are $\simeq n^\delta$.
\end{cor}

This follows from the theorem (namely the already-proved implication (\ref{xcarrq}$^+$)$\Rightarrow$(\ref{xelasy})), along with the commensurability invariance of the systolic and normal systolic growth, the latter being checked in Lemma \ref{nsgcoi}.

In general, let us provide an upper bound on the systolic (and normal systolic growth) of an arbitrary finitely generated torsion-free nilpotent group.

\begin{prop}\label{uppersys}
Let $\Gamma$ be a finitely generated torsion-free nilpotent group, of growth $\simeq n^\delta$ and nilpotency length $c$. Let $\g$ be its rational Lie algebra. Then the systolic and normal systolic growth of $\Gamma$ are $\preceq n^D$, where $D=c\dim([\g,\g])+\dim(\g/[\g,\g])\le c\dim(\g)$.
\end{prop}
\begin{proof}[Proof (sketched)]
Let $\g_1$ be a supplement subspace of $[\g,\g]$ in $\g$, and fix a basis of $\g$ adapted to the decomposition $\g=\g_1\oplus [\g,\g]$, such that the Lie algebra constants are divisible by $c!$. This implies that if we define $\Lambda'_n$ as $\g_1(n\mathbf{Z})\oplus [\g,\g](n^c\mathbf{Z})$, then by the Baker-Campbell-Hausdorff formula (whose denominators divide $c!$), $\Lambda_n=\exp(\Lambda'_n)$ is a subgroup (hence a lattice) in $G$. The index of $\Gamma_n$ is $n^D$ and its systole and normal systole (in $G$) are $\sim n$. Hence $\Gamma_n\cap\Gamma$ has index $\simeq n^D$ and normal systole $\sim n$. This proves that the systole and normal systole of $\Gamma$ are $\preceq n^D$.
\end{proof}

Note that this upper bound is optimal when $c\le 2$, but far from optimal when $c>2$ and $\g$ is Carnot. I do not know if it is close to optimal in the ``worst" cases.

\subsection{The geometric part of the proof}\label{gpart}

We now complete the proof of Theorem \ref{corlat2}, by proving (\ref{xecovol})$\Rightarrow$(\ref{xcarr}). We note that here no particular choice of $\mathbf{Q}$-structure on $\g$ is relevant. Hence we use the notation as in Theorem \ref{corlat}, and will assume (\ref{ecovol}) (which is equivalent to (\ref{xecovol} of Theorem \ref{corlat2})), namely that $G$ admits a sequence $(\Gamma_n)$ of lattices with systole $u_n\to\infty$ and covolume $\preceq u_n^\delta$, and aim at proving that the real Lie algebra $\g$ is Carnot.

We need some further definitions pertaining to the geometry of lattices. For the moment, let $G$ be a group endowed with a left-invariant distance. Let $B_r$ be the closed $r$-ball in $G$.

\begin{defn}Given a subgroup $\Gamma$ of $G$, define its {\em packing} $\pack_G(\Gamma)$ to be $\sup_{g\in G}d(g,\Gamma)\in [0,\infty]$. Also define its {\em generating radius} $\ger_G(\Lambda)$ as the infimum of the $r$ such that $\Lambda$ is generated by $\Lambda\cap B_r$. 
\end{defn}

All these definitions are understood with the usual conventions: the infimum of the empty set and the supremum of an unbounded subset of positive reals are $+\infty$.

\begin{lem}\label{page}
Let $G$ be a locally compact group with a continuous left-invariant geodesic distance $d$ and let $\Gamma$ be a cocompact lattice. Then $\ger_G(\Gamma)\le 2\pack_G(\Gamma)$; moreover any element $\gamma\in\Gamma$ with $d(\gamma,1)\le n$ is a product of $n$ elements of $B_{2\pack_G(\Gamma)+1}\cap\Gamma$.
\end{lem}
\begin{proof}
Fix an integer $m\ge 1$. Given $\gamma\in\Gamma$ with $d(\gamma,1)\le n$, consider a geodesic joining $\gamma$ to 1. On this geodesic choose points $1=x_0,\dots,x_{mn}=\gamma$ with $d(x_{i-1},x_i)\le 1/m$ for all $i=1\dots k$. There is $\gamma_i$ in $\Gamma$ with $d(x_i,\gamma_i)\le\pack_G(\Gamma)$ for all $i$, where we choose $\gamma_0=1$ and $\gamma_k=\gamma$. Hence $\gamma=\prod_{i=1}^k\gamma_{i-1}^{-1}\gamma_i$, and $d(1,\gamma_{i-1}\gamma_i)\le 2\pack_G(\Gamma)+1/m$ for all $i$. Hence $\ger_G(\Gamma)\le 2\pack_G(\Gamma)+1/m$; since this holds for all $m$ we deduce $\ger_G(\Gamma)\le 2\pack_G(\Gamma)$; on the other hand taking $m=1$ in the above argument shows that $\gamma$ is a product of $n$ elements from $B_{2\pack_G(\Gamma)+1}\cap\Gamma$.
\end{proof}

\begin{lem}\label{packger}
Let $V$ be a Euclidean space and $\Lambda$ a lattice (in this case, $\ger_V(\Lambda)$ is often denoted $\lambda_{\dim(V)}(\Gamma)$ in the literature).  Then 
\[\frac2{\dim(V)}\pack_V(\Lambda)\le\ger_V(\Lambda)\le 2\pack_V(\Lambda).\]
\end{lem}
\begin{proof}
The right-hand inequality is borrowed from Lemma \ref{page}. For the left-hand inequality, $V$ has a basis $(e_i)$ with $e_i\in\Lambda$ and $\|e_i\|\le\ger_V(\Lambda)$. If $x\in V$, we write $x=\sum\alpha_ie_i$; hence we can decompose $x=w+y$ with $w\in\Lambda$ and $y=\sum\beta_ie_i$ with $|\beta_i|\le 1/2$ for all $i$. Hence $\|y\|\le \dim(V)\ger_V(\Lambda)/2$, whence $\pack_V(\Gamma)\le \dim(V)\ger_V(\Lambda)/2$. 
\end{proof}

Assume now that $G$ is a simply connected nilpotent Lie group, endowed with a left-invariant Riemannian metric. Guivarch \cite{Guiv} established that the growth rate of $G$ and of its lattices is $\simeq n^\delta$, where $\delta$ is characterized in terms of the lower central series by $\delta=\sum_{i\ge 1}i\dim(\g^{(i)})$. 

The arithmeticity of lattices (see \cite{Ragh}) implies in particular that for every lattice in $G$, its projection on $G/[G,G]$ is also a lattice. We endow $V=G/[G,G]$ with the Euclidean metric defined by identifying $\g/[\g,\g]$ with the orthogonal of $[\g,\g]$. Let $p:G\to G/[G,G]$ be the projection, which is 1-Lipschitz. Let $\lambda$ be a left Haar measure on $G$.
 
\begin{lem}\label{papa}
There exists a constant $C$ (depending only on $G$ and its Riemannian metric) such that for every lattice $\Gamma$ in $G$ and $\Lambda=p(\Gamma)$, we have $\pack_G(\Gamma)\le C\pack_{V}(\Lambda)$.
\end{lem}
\begin{proof}
We argue by induction on the nilpotency length $c$ of $G$. If $c=1$ the result is trivial (with $C=1$). Otherwise, the $c$-iterated commutator induces an alternating multilinear form from $V^c$ to $G^{(c)}$, and more precisely a surjective linear map $F$ from $\Lambda^cV$ onto $G^{(c)}$. If we endow both $V$ and $G^{(c)}$ with their intrinsic Riemannian (Euclidean) metric, There exists a constant $C_0$ such that $F(v_1,\dots,v_c)\le C_0\prod_{i=1}^c\|v_i\|$ for all $v_1,\dots,v_c\in V$.

Note that $F(\Lambda^{\ot c})$ is a lattice in $G^{(c)}$, of finite index in $\Gamma\cap G^{(c)}$. Moreover, it is generated by the image of the generators of $\Lambda^{\ot c}$, and therefore is generated by elements of norm $\le C_0\ger_V(\Lambda)^c$. Thus, using twice Lemma \ref{packger}, we successively obtain $\pack_{G^{(c)}}(F(\Lambda^{\ot c}))\le C_1\ger_V(\Lambda)^c$ for some constant $C_1=(\dim G^{(c)})C_0/2$ and then, $\pack_{G^{(c)}}(F(\Lambda^{\ot c}))\le C_2\pack_V(\Lambda)^c$ with $C_2=2^cC_1$. 

On the other hand, denote by $p'$ the projection $G\to G'=G/G^{(c)}$. If $\Gamma'=p(\Gamma)$, we have, by induction, $\pack_{G'}(\Gamma')\le C_3\pack_V(\Lambda)$ for some constant $C_3$ depending only on $G$ and its fixed Riemannian metric. Thus if $x\in G$, there exists $\gamma\in\Gamma$ such that $d(p'(x),p'(\gamma))\le C_3\pack_V(\Lambda)$.

If we lift a minimal geodesic joining $1$ to $p(x^{-1}\gamma)$, we obtain $y\in G$ such that $p(y)=p(\gamma^{-1}x)$ and $d(1,y)\le C_3\pack_V(\Lambda)$. Since $y^{-1}\gamma^{-1}x\in G^{(c)}$, there exists $\gamma'\in F(\Lambda^{\ot c})\subset\Gamma$ with $d_{G^{(c)}}(\gamma'^{-1}y^{-1}\gamma^{-1}x,1)\le C_2\pack_V(\Lambda)^c$. Here $d_{G^{(c)}}$ is the intrinsic distance of $G^{(c)}$, which by Guivarch's estimates is distorted in such a way that $d(w,1)\le C_4d_{G^{(c)}}(w,1)^{1/c}$ for all $w\in G^{(c)}$. Hence, writing $s=\gamma'^{-1}y^{-1}\gamma^{-1}x$, we have $d(1,s)\le C_4C_2^{1/c}\pack_V(\Lambda)$.

We have $x=\gamma y\gamma's=\gamma\gamma'ys$, because $\gamma'$ is central. Hence we have 
\[d(ys,1)\le  C\pack_V(\Lambda);\quad C=C_3+C_4C_2^{1/c}\]
thus $d(x,\Gamma)\le C\pack_V(\Lambda)$ and accordingly $\pack_G(\Gamma)\le C\pack_V(\Lambda)$.
\end{proof}

\begin{lem}\label{palb}
For every lattice $\Gamma$ in $G$ with systole $\ge 2r+1$, we have, denoting again $\Lambda=p(\Gamma)\subset V=G/[G,G]$, the following lower bound on its covolume
\[\covol_G(\Gamma)\ge \frac{\pack_V(\Lambda)\lambda(B_r)}{2r+1}.\]
\end{lem}
\begin{proof}
Define a possibly finite sequence of cosets $W_i$ of $\Lambda$ by $W_1=\Lambda$, and, assuming $W_1,\dots,W_i$ are defined, if $d(x,\bigcup_{1\le j\le i}W_j)<2r+1$ for all $x\in V$, then stop; otherwise there exists, by connectedness, $x\in V$ such that $d(x,\bigcup_{1\le j\le i}W_j)=2r+1$ and we define $W_{i+1}=x+\Lambda$.

Since the $W_i$ are at pairwise distance $\ge 2r+1$, the process stops, say at $i=k_r$. Since for $2\le i\le k_r$ every point in $W_i$ is at distance $2r+1$ to a point in $\bigcup_{j<i}W_j$ and every point in $V$ is at distance $\le 2r+1$ to a point in $\bigcup_iW_i$, it follows that every point in $V$ is at distance $\le k_r(2r+1)$ of some point in $\Lambda$. In other words, $\pack_V(\Lambda)\le k_r(2r+1)$.

Fix $x_i\in W_i$ and lift it to some element $g_i\in G$. Define $X=\bigcup_{1\le i\le k_r} x_iB_r$. This is a disjoint union, since the $x_i$ are at pairwise distance $\ge 2r+1$. Moreover, the $X\gamma$ for $\gamma\in\Gamma$ are pairwise disjoint: indeed if $g_ib\gamma=g_jb'\gamma'$ with $b,b'\in B_r$ and $\gamma\neq\gamma'\in\Gamma$, then, projecting, we obtain $x_i-x_j+p(b)-p(b')=p(\gamma^{-1}\gamma')\in\Lambda$. 
Since $\|p(b)-p(b')\|\le 2r$, this forces $i=j$. Thus $g_i=g_j$, hence $b\gamma=b'\gamma'$. Hence $b^{-1}b'=\gamma\gamma'^{-1}\in\Gamma$; since the systole of $\Gamma$ is $\ge 2r+1$, this implies $\gamma=\gamma'$, contradiction. This proves that the covolume of $\Gamma$ is at least equal to the volume of $X$, and hence is $\ge k_r\lambda(B_r)$. 

Combining both inequalities yields the lemma.
\end{proof}

\begin{proof}[Conclusion of the proof of (\ref{ecovol})$\Rightarrow$(\ref{carr})]

Let now $(\Gamma_n)$ be a sequence of lattices in $G$, satisfying $\sys(\Gamma_n)\ge 2u_n+1$ and $\covol(\Gamma_n)\preceq u_n^\delta$. Define $\Lambda_n=p(\Gamma_n)$ as the projection of $\Gamma_n$ on $V=G/[G,G]$. 

We first claim that we have $\pack_V(\Lambda_n)\preceq u_n$.
Indeed, we have, by Lemma \ref{palb}, $\covol(\Gamma_n)\ge\pack_V(\Lambda_n)\lambda(B_{u_n})/(2u_n+1)$. Since by assumption $\covol(\Gamma_n)\simeq\lambda(B_{u_n})\simeq u_n^\delta$, we deduce that $\pack_V(\Lambda_n)\preceq u_n$, proving the claim.

Lemma \ref{papa} combined with the above claim implies that $\pack_G(\Gamma_n)\preceq u_n$, say $\pack_G(\Gamma_n)\le (Cu_n-1)/2$ (for $n$ large enough).
It follows from Lemma \ref{page} that $\Gamma_n$ is generated by the elements in $B_{Cu_n}\cap\Gamma_n$, in such a way that for any integer $R\ge 1$, any element in the $B_{Ru_n}\cap\Gamma_n$ is a product of at most $R$ elements in $B_{Cu_n}\cap\Gamma_n$.

If we divide the distance in $G$ by $u_n$, the lattice $\Gamma_n$ endowed with the resulting distance has the property that its systole is $\ge 2+1/u_n$ and that every element in the $R$-ball is product of at most $R$ elements in the $C$-ball, and the packing of $\Gamma_n$ in $(G,(1/u_n)d)$ is bounded independently of $n$.

By Pansu's thesis \cite{PanTh}, the $(G,(1/n)d$) converge in the sense of Gromov-Hausdorff to a (real) Carnot simply connected nilpotent Lie group endowed with a Carnot-Carath\'eodory metric, with the same dimension as $G$ (that $H$ is isometric to a Carnot group is due to Pansu; that $H$ inherits the group law as limit of the laws from $G$ is proved in \cite{CoI}). Denote by $B_H(r)$ the closed $r$-ball in $H$.

Fix any non-principal ultrafilter $\omega$ on the positive integers. The metric ultralimit $\Xi$ of the sequence $(\Gamma_n,(1/u_n)d)$ is a discrete subset of $H$, with systole $\ge 2$, with the property that any element of $\Xi\cap B_H(R)$ is a product of at most $R$ elements of $\Xi\cap B_H(C)$, for all $R\ge 1$, and any element of $H$ is at bounded distance to some element of $\Xi$. The fact that $\Xi$ is a subgroup follows from the refinement in \cite{CoI} of Pansu's result mentioned above. Thus $\Xi$ is a lattice in $H$. 

Recall that a marked group on $k$ generators is a group endowed with a map (called marking) $s$ from $\{1,\dots,k\}$, whose image generates the group, and a net $(M_i,s_i)$ converges to $(M,s)$ if, denoting by $N_i$ (resp.\ $N$) the kernel of the unique homomorphism $F_k\to M_i$ extending $s_i$ (resp.\ $F_k\to M$ extending $s$), we have the convergence $\mathbf{1}_{N_i}\to\mathbf{1}_N$ pointwise on the set of functions $F_k\to\{0,1\}$; see \cite{CG} for more details. 

The sequence $(\#(B(Cu_n)\cap\Gamma_n))$ being bounded, let $k$ be an upper bound and for each $n$, choose a surjective map $s_n$ from $\{1,\dots k\}$ onto $B(Cu_n)\cap\Gamma_n$. 

Since $B(Cu_n)\cap\Gamma_n$ generates $\Gamma_n$, this provides a marking of $\Gamma_n$. Define $s(i)=\lim^{n\to \omega} s_n(i)\in\Xi$ for $1\le i\le k$. Then $s$ defines a marking of $\Xi$: indeed every element of the $R$-ball in $(\Gamma_n,(1/u_n)d)$ is a product of $\le R$ elements in the $C$-ball; this fact passes to the ultralimit to show that every element in the $R$-ball of $\Xi$ is a product of $\le R$ elements of $S$, and thus the image of $s$ generates $\Xi$.

A straightforward argument (using that these groups are uniformly discrete) then shows that $(\Gamma_n,s_n)$ tends to $(\Xi,s)$ for the topology of marked groups.

Since $\Xi$ is finitely presented, eventually $\Gamma_n$ lies as a quotient of $\Xi$ (by \cite[Lemma 2.2]{CG}), in the sense that there exists $I\in\omega$ such that for all $n\in I$, the kernel of $F_k\to\Gamma_n$ contains the kernel of $F_k\to\Xi$. Since both $\Xi$ and $\Gamma_n$ are torsion-free of the same Hirsch length, we deduce that $\Gamma_n$ is isomorphic to $\Xi$ for every $n\in I$. The rigidity of nilpotent lattices \cite{Ragh} asserts that if two simply connected nilpotent Lie groups admit isomorphic lattices, then they are isomorphic. It follows that $G$ is isomorphic to $H$, and hence that $G$ is Carnot (i.e., $\g$ is Carnot over $\mathbf{R}$).

(Note that the fact that $\Gamma_n$ is generated by elements of length $\preceq u_n$ ---and hence bounded length after rescaling--- played a crucial role: otherwise the ultralimit of the $(\Gamma_n,(1/u_n)d)$ could have been of Hirsch length less that that of $G$, yielding no conclusion.)
\end{proof}

\appendix

\section{Cohopfian does not pass to finite index}\label{apfi}
It was asserted in \cite{CM} that it is an open question whether cohopfian is inherited by subgroups of finite index; however, it seems that this question is already settled in the negative: a general criterion for being cohopfian, for finitely presented groups with infinitely many ends, with cohopfian vertex groups, was provided by Delzant and Potyagailo. It turns out that this criterion can be applied to virtually free groups. The general theorem being a bit technical, let us provide an easier instance, with a self-contained proof.

\begin{prop}
Let $A,B$ be cohopfian groups with Property FA, with a common cohopfian subgroup $C$. Suppose that $C$ is equal to its normalizer in both $A$ and $B$. Then the amalgam $G=A*_CB$ is cohopfian. In particular, there exists some cohopfian group among non-elementary virtually free finitely generated groups, and being cohopfian does not pass to finite index subgroups. 
\end{prop}

Recall that a group has Property FA if each of its actions on a tree fixes a vertex or an edge. For instance, finite groups have Property FA, and this is enough for our purposes.

\begin{proof}
We can suppose that $A\neq C\neq B$ since otherwise the result is trivial.

Let $G$ act on its Bass-Serre tree $T$, with a fundamental domain consisting of an edge $e$ with stabilizer $C$ and vertices $a$ and $b$ with stabilizers $A$ and $B$ respectively. We observe that the only vertices fixed by $C$ are $a$ and $b$. Indeed, otherwise it would fix another neighbor of $a$ or $b$; but the self-normalization assumption rules this out.

Let $\phi$ be an injective endomorphism of $G$. Then $\phi(A)$ and $\phi(B)$ fix vertices $a'$ and $b'$ respectively, and in particular $\phi(C)$ fixes the segment $[a',b']$. If $a'=b'$, we obtain an embedding of $G$ into either $A$ or $B$, which contradicts that $A$ and $B$ are cohopfian. Otherwise, after possibly composing $\phi$ with an inner automorphism, we can suppose that $e$ is contained in $[a',b']$. Then $\phi(C)$ fixes $e$, and hence $\phi(C)\subset C$. Since $C$ is cohopfian, we deduce from the previous observation that $\phi(C)=C$. In particular, since $\phi(C)$ fixes $[a',b']$, we deduce that $\{a,b\}=\{a',b'\}$. Replacing $\phi$ by $\phi^2$ if necessary, we have $(a,b)=(a',b')$. Thus $\phi(A)\subset A$, and $\phi(B)\subset B$; by the cohopfian assumption, we have $\phi(A)=A$ and $\phi(B)=B$, and the surjectivity of $\phi$ follows.

The second statement follows by picking, for instance, the double $S_3*_{S_2}S_3$, where $S_n$ is the symmetric group on $n$ elements (as a finite group, $S_3$ has Property FA). This group is cohopfian by the first statement, but admits a non-abelian free subgroup of finite index (of index 6), which is not cohopfian.
\end{proof}

\begin{rem}
In contrast, being Hopfian, for finitely generated groups (or more generally groups with finitely many subgroups of each given index) is inherited by finite index subgroups \cite[Co.\ 2]{Hir}.

On the other hand, I do not know if the cohopfian property passes to overgroups of finite index, including in the case of finitely generated groups.
\end{rem}



\begin{thebibliography}{KM98b}

\bibitem[AC]{AC} J. Ancochea, R. Campoamor.
Characteristically nilpotent Lie algebras: a survey. 
Extracta Math. 16 (2001), no. 2, 153--210. 

\bibitem[AOR]{AOR} S. Albeverio, B. Omirov, I. Rakhimov. Classification of 4-dimensional nilpotent complex Leibniz algebras. Extracta Math. 21(3) 197--210 (2006).

\bibitem[BC]{BC} K. Bou-Rabee and Y. Cornulier.  Systolic growth of linear groups. Proc. Amer. Math. Soc. 144 (2016), no. 2, 529--533.


\bibitem[Bel]{Bel} I. Belegradek. On co-Hopfian nilpotent groups. Bull. London Math. Soc. 35(6) (2003) 805--811.

\bibitem[BM]{BM} K. Bou-Rabee and D. McReynolds. Asymptotic growth and least common multiples in groups. Bull. London Math. Soc. 43(6) (2011) 1059--1068.

\bibitem[BS1]{BS} K. Bou-Rabee and B. Seward. Arbitrarily large residual finiteness growth. J. Reine Angew. Math. 710 (2016), 199--204.

\bibitem[BS2]{BSer} A. Borel and J-P. Serre. Th\'eor\`emes de finitude en cohomologie galoisienne. Comment. Math. Helv. 39 (1964), 111--169.

\bibitem[BT]{BT} A. Borel and J. Tits. Groupes r\'eductifs. \newblock Publ. Math. IHES 27, 55-150, 1965.

\bibitem[CG]{CG}
C.~Champetier and V.~Guirardel.
\newblock Limit groups as limits of free groups.
\newblock Israel J. Math., 146 (2005) 1--75.

\bibitem[CM]{CM} A. Cain, V. Maltcev. Hopfian and co-hopfian subsemigroups and extensions. Demonstr. Math. 47 (2014), no. 4, 791--804.


\bibitem[Con]{Con} B. Conrad. Reductive group schemes. Pages 93-444 in: Autour des sch\'emas en groupes. Vol. I. A celebration of SGA3. Lecture notes from the Summer School held at the CIRM, Luminy, August 29--September 9, 2011. Panoramas et Synth\`eses, 42/43. SMF, Paris, 2014.

\bibitem[Cor1]{CoI} Y. Cornulier. Asymptotic cones of Lie groups and cone equivalences. Illinois J. Math. 55(1) (2011), 237--259.

\bibitem[Cor2]{CorInd} Y. Cornulier. Commability and focal locally compact groups. (2013) To appear in Indiana Univ. Math. J. ArXiv:1306.4194.

\bibitem[Dan]{Dan} S.D. Dani. Nilmanifolds with Anosov automorphism. J. London Math. Soc. (2) 18 (1978) 553--559.

\bibitem[DD]{DD} K. Dekimpe and J. Der\'e. Expanding maps and non-trivial self-covers on infra-nilmanifolds. Topol. Methods Nonlinear Anal. 47, Number 1 (2016), 347--368.

\bibitem[Der]{De} J. Der\'e. Gradings on Lie algebras with applications to
infra-nilmanifolds. ArXiv 1410.3713v3 (2016)

\bibitem[DL]{DL} K. Dekimpe and K.B. Lee. Expanding maps on infra-nilmanifolds of homogeneous type.
Trans. Amer. Math. Soc., 2003, 355 (3), pp. 1067--1077.

\bibitem[DP]{DP} T. Delzant and L. Potyagailo. Endomorphisms of Kleinian groups. Geom. Funct. Anal. 13 (2003) 396--436.

\bibitem[Fav]{Fav} G. Favre.
Syst\`eme de poids sur une alg\`ebre de Lie nilpotente.
Manuscripta Math. 9 (1973), 53--90. 

\bibitem[Gra]{graaf} W. de Graaf. Classification of 6-dimensional nilpotent Lie algebras over fields of characteristic not 2. Journal of Algebra 309 (2007) 640--653.


\bibitem[Gro]{Gro} M. Gromov. Systoles and intersystolic inequalities. Actes de la table ronde de g\'eom\'etrie diff\'erentielle (Luminy, 1992), 291--362, S\'emin. Congr., 1, Soc. Math. France, Paris, 1996.

\bibitem[Gui]{Guiv} Y. Guivarc'h. Croissance polynomiale
et p\'eriodes des fonctions harmoniques. Bull. Soc. Math. France
101 (1973) 333--379.

\bibitem[Hir]{Hir} R. Hirshon. Some theorems on hopficity. Trans. Amer. Math.
Soc. 141 (1969) 229--244.

\bibitem[Joh]{Joh} R.W. Johnson. Homogeneous Lie Algebras and Expanding Automorphisms. Proc. Amer. Math. Soc. 48(2) (1975), 292--296.

\bibitem[Lau]{Lau}
J. Lauret. Examples of Anosov diffeomorphisms. 
J. Algebra 262 (2003), no. 1, 201--209; corrigendum J. Algebra 268 (2003), no. 1, 371--372. 


\bibitem[Leg]{Le} G. Leger. Derivations of Lie algebras III. Duke Math. J.
Volume 30, Number 4 (1963), 637--645.

\bibitem[Mag]{Mag} L. Magnin. 
Determination of 7-dimensional indecomposable nilpotent complex Lie algebras by adjoining a derivation to 6-dimensional Lie algebras
Algebras and Representation Theory, vol.13, Number 6, p. 723--753, 2010.
 
\bibitem[NP]{NP} V. Nekrashevych and G. Pete.
Scale-invariant groups. 
Groups Geom. Dyn. 5 (2011), no. 1, 139--167. 

\bibitem[Pan1]{PanTh} P. Pansu. Croissance
des boules et des g\'eod\'esiques ferm\'ees dans les
nilvari\'et\'es. Ergodic Theory Dyn. Syst. 3 (1983) 415--445. 
 
\bibitem[Pan2]{Pan} P. Pansu. M\'etriques de Carnot-Carath\'eodory et quasiisom\'etries des espaces sym\'etriques de rang un, Ann. of Math. 129(1) (1989), 1--60. 
 
\bibitem[Rag]{Ragh} 
M.S. Raghunathan.
Discrete subgroups of Lie groups. 
Ergebnisse der Mathematik und ihrer Grenzgebiete 68. Springer-Verlag, 1972.  
 
\bibitem[Sch]{Sch} R. Schafer. An introduction to nonassociative Algebras. Dover, 1966.
 
\bibitem[Sul]{Sul} D. Sullivan. Infinitesimal computation in topology. Publ. IHES 47 (1977) 269--331.
 
\bibitem[Ver]{Ver} M. Vergne. Cohomologie des alg\`ebres de Lie nilpotentes. Applications a l'\'etude de la
vari\'et\'e des alg\`ebres de Lie nilpotentes, Bull. Soc. Math. France 98 (1970), 81--116.
\end{thebibliography}
\end{document}